%% file: WrapPong.tex
\documentclass{amsart}
\usepackage{amsmath,enumitem, amssymb, graphicx, epsfig, verbatim, psfrag, color,amscd,stmaryrd} 
\usepackage{tikz,mathabx}
\usetikzlibrary{arrows}

\def\endproof{\relax\ifmmode\expandafter\endproofmath\else
  \unskip\nobreak\hfil\penalty50\hskip.75em\hbox{}\nobreak\hfil\bull
  {\parfillskip=0pt \finalhyphendemerits=0 \bigbreak}\fi}
\def\endproofmath$${\eqno\bull$$\bigbreak}
\def\bull{\vbox{\hrule\hbox{\vrule\kern3pt\vbox{\kern6pt}\kern3pt\vrule}\hrule}}

\newcommand\doms{\mathcal D}

\newcommand\Sym{\mathrm{Sym}}
\newcommand\Mas{\mu}
\newcommand\Ta{{\mathbb T}_{\alpha}}
\newcommand\Tb{{\mathbb T}_{\beta}}
\newcommand\Tc{{\mathbb T}_{\gamma}}
\newcommand\CF{\mathit{CF}}

\newcommand\ModFlow{\mathcal M}

\newcommand\States{\mathbf S}
\newcommand\Habc{{\mathcal H}_{\alpha\beta\gamma}}
\newcommand\Hab{{\mathcal H}_{\alpha\beta}}
\newcommand\Hbc{{\mathcal H}_{\beta\gamma}}
\newcommand\Hac{{\mathcal H}_{\alpha\gamma}}
\newcommand\HD{\mathcal H}

\newcommand\Pong[2]{{\mathcal P}(#1,#2)}

\newcommand\OneHalf{\frac{1}{2}}
\newcommand\weight{\mathfrak w}

\newcommand\Alg\AlgA
\newcommand\Blg\AlgB

\newcommand\Ainf{\mathcal A}
\newcommand\Ainfty\Ainf

\newcommand\alphas{\mathbf{\alpha}}
\newcommand\betas{\mathbf{\beta}}
\newcommand\gammas{\mathbf{\gamma}}

\newcommand\Zmod[1]{{\mathbb Z}/{#1}{\mathbb Z}}
\newtheorem{thm}{Theorem}[section]

\newtheorem{lemma}[thm]{Lemma}
\newtheorem{prop}[thm]{Proposition}
\newtheorem{defn}[thm]{Definition}

\numberwithin{equation}{section}

\newcommand\Idemp[1]{{\mathbf{I}}_{#1}}

\newcommand\x{\mathbf x}
\newcommand\y{\mathbf y}

\newcommand{\AlgA}{{\mathcal A}}
\newcommand{\AlgB}{{\mathcal B}}

\newcommand{\Mor}{\mathrm{Mor}}

\newcommand\z{\mathbf z}

\renewcommand{\L}{{\mathfrak {L}}}

\newcommand{\C}{\mathbb C} \newcommand{\Z}{\mathbb Z}   \newcommand{\R}{\mathbb R}

\newcommand{\CFm}{{\rm {CF}} ^-}

\newcommand\Field{\mathbb F}

\newcommand\cross{\mathrm{cross}}
\newcommand\Cross{\mathrm{Cross}}
\newcommand\LiftS{\widetilde S}

\DeclareMathOperator{\Id}{Id}

\begin{document}
\title{The pong algebra and the wrapped Fukaya category}

\begin{abstract}
  The aim of this paper is to identify the pong algebra
  defined in our earlier work
  with a certain endomorphism algebra in the wrapped Fukaya category
  of the symmetric product of a disk.
\end{abstract}

\author[Peter S. Ozsv\'ath]{Peter Ozsv\'ath}
\thanks {PSO was partially supported by NSF grant number DMS-1708284, DMS-2104536, and the Simons Grant {\em New structures in low-dimensional topology}.}
\address {Department of Mathematics, Princeton University\\ Princeton, New Jersey 08544} 
\email {petero@math.princeton.edu}

\author[Zolt{\'a}n Szab{\'o}]{Zolt{\'a}n Szab{\'o}}
\thanks{ZSz was supported by NSF grant number DMS-1904628
  and the Simons Grant {\em New structures in low-dimensional topology}.}
\address{Department of Mathematics, Princeton University\\ Princeton, New Jersey 08544}
\email {szabo@math.princeton.edu}

\input{intro}

\input{wrap}
\input{pongrev}
\input{heegpong}

\input{triangles}

\input{ngons}
\input{further}

\bibliographystyle{plain}
\bibliography{biblio}

\end{document}

%% file: intro.tex
\maketitle
\section{Introduction}

In~\cite{Pong}, we introduced a differential graded algebra, the {\em
  pong algebra}, which is an enrichment of the strands algebra
  from~\cite{InvPair}. The aim of this note is to identify this
  algebra with the endomorphism algebra in a wrapped, relative Fukaya
  category, of a distinguished set of objects.

Consider $\C$, equipped with punctures at the points ${\mathbf
  P}=\{j+\OneHalf\}_{j=0}^{m-1}$, and let $M=\C\setminus {\mathbf P}$.
Consider $m-1$ disjoint vertical lines $e_j=j\times \R $ for
$j=1,\dots,m-1$.

Given any $k$-element subsequence $1\leq x_1<\dots<x_k\leq m-1$, which
we call a {\em $k$-idempotent state}, there is a corresponding
Lagrangian ${\Lambda}_{\x}=e_{x_1}\times\dots\times e_{x_k}\subset
\Sym^k(M\subset\Sym^k(\C)$

The manifold $\Sym^k(M)$ is a Liouville manifold, and the Lagrangians
$\Lambda_\x$ are exact and conical at infinity.  One can then consider
the {\em wrapped Fukaya category} of introduced by Abouzaid and
Seidel~\cite{AbouzaidSeidel} relative to the divisor ${\mathbf
P}\times \Sym^{k-1}(\C)$; see also~\cite{AbouzaidCriterion,AurouxBeginner}. We
will consider the endomorphism algebra of the set of objects
$\{\Lambda_\x\}_{\x}$ indexed by $k$-idempotent states.

This is an A-infinity algebra $A$ over $\Field[v_1,\dots,v_m]$,
equipped with idempotents $\Idemp{\x}$  correspondingto the idempotent states $\x$,
and $\Idemp{\x}\cdot A \cdot \Idemp{\y}$ is given by the chain complex
$\CF(\phi_H(\Lambda_\x),\Lambda_\y)$ (again, relative to the divisor
${\mathbf P}$). The composition is given by the chan map
\[ \circ\colon  \Mor(\Lambda_{x_2},\Lambda_{\x_3}) \otimes \Mor(\Lambda_{\x_1},\Lambda_{\x_2})\to \Mor(\Lambda_{\x_1},\Lambda_{\x_3}),\]
is specified by the following diagram
\begin{equation}
        \label{lem:TriangleMap}
        \begin{CD}
  \CF(\phi_H(\Lambda_{\x_1}),\Lambda_{\x_2})\otimes \CF(\phi_H(\Lambda_{\x_2}),\Lambda_{\x_3}) @>{\circ}>>\CF(\phi_H(\Lambda_{\x_1}),\Lambda_{\x_3}) \\
  @V{(\phi_H)_*\otimes \Id}VV @A{\sigma}AA \\
  \CF(\phi^2_H(\Lambda_{\x_1}),\phi_H(\Lambda_{\x_2})\otimes \CF(\phi_H(\Lambda_{\x_2}),\Lambda_{\x_3}) @>>>\CF(\phi^2_H(\Lambda_{\x_1}),\Lambda_{\x_3}), \\
  \end{CD}
\end{equation}
where $\sigma$ is a map induced from the Liouville flow, and the
bottom arrow is induced by counting pseudo-holomorphic triangles. The 
base ring is $\Field[v_1,\dots,v_m]$; and for the map in the bottom arrow,
each holomorphic triangle
$u$ is counted as a monomial $v_1^{n_{p_1}(u)}\dots v_m^{n_{p_m}(u)}$.

We can take the homology to obtain an ordinary category, $H({\mathcal
  C})$; or alternatively, we can consider the $A_{\infty}$ category,
where the higher compositions are defined by counting
pseudo-holomorphic polygons.

Our aim is to prove the following:

\begin{thm}
  \label{thm:IdentifyPong}
  The pong algebra is isomorphic to the endomorphism algebra, in the
  wrapped relative Fukaya category, of the objects
  $\{\Lambda_\x\}_{\x}$.
\end{thm}

This proof is in the spirit of Auroux~\cite{Auroux}; see
also~\cite{TorusAlg}. Indeed, the proof is a fairly straightforward
application of a suitable Heegaard diagram, which can be thought of as
the analogue of the ``Auroux-Zarev piece'' for the pong algebra;
see~\cite{Auroux,Zarev,HomPairing}.

It is interesting to compare the results herein with the results
of~\cite{LaudaLicataManion,ManionRouquier}, and the constructions
of~\cite{PetkovaVertesi,EllisPetkovaVertesi,Zibrowius,KotelskiyWatsonZibrowius}.

Our interest in the pong algebra stems from our
goal of understanding knot Floer homology, a topic which we do not
discuss in the present paper, but hope to return to in future
work~\cite{NextPong}. 

{\bf{Acknowledgements:}} The authors wish to thank Denis Auroux,
Robert Lipshitz, Dylan Thurston, and Andy Manion for interesting
conversations.

%% file: wrap.tex
\section{The wrapped Fukaya category}
\label{sec:Wrap}

\subsection{The symmetric product}
We start with some of the geometric setup, as explained
in~\cite{AbouzaidSeidel,CieliebakEliashberg,AurouxBeginner}.

\begin{defn}
  A {\em Liouville domain} is a $2n$-dimensional manifold with
  boundary, equipped with a one-form $\lambda$ such that
  $\omega=d\lambda$ is symplectic, and the dual vector field, called
  the {\em Liouville vector field}, $Z$ characterized by
  $i_Z\omega=\lambda$ points strictly outwards along $\partial M$.
\end{defn}

A special case of a Liouville domain is a Stein manifold, which is a
complex manifold $(V,J)$, equipped with a proper, smooth function
$\phi\colon V \to \R$, which is {\em strictly psudo-plurisubharmonic};
i.e. for which the two-form $\omega=-dd^{\C}\phi$ is symplectic and
$J$-compatible. (Here, $d^\C\phi=d\phi\circ J$.) In this case,
$\lambda=-d^{\C}\phi$; i.e. $d\phi=\lambda\circ J$.

Our basic example is the following. Let $A$ denote the infinite
cylinder $\R\times S^1\cong \C\setminus 0$. Let $(t,\theta)$ denote
the coordinates with respect to the parameterization $\R\times S^1$;
so that the isomorphism $\R\times S^1\cong \C\setminus 0$ is given by
$(t,\theta)\mapsto e^{t+i\theta}$; sometimes we write $r=e^{t}$.  The
function $\log(r)^2$ is strictly pluri-subharmonic, with $\omega=2
dt\wedge\theta$.    Let
$H=r^2$. The Hamiltonian flow for $H$, written $\Phi\colon \R\times A \to A$, is
given by
\begin{equation}
  \label{eq:HamiltonianFlow}
  \Phi(s,t,\theta)=(t,2 s t \theta).
\end{equation}

The Liouville flow, written $\Psi\colon \R\times A \to A$, is given by
\begin{equation}
  \label{eq:LiouvilleFlow}
  \Psi(s,t,\theta)=(t e^s,\theta).
\end{equation}

To a Liouville domain Abouzaid and Seidel associate an $A_\infty$
category, the {\em wrapped Fukaya category}. Objects are Lagrangian
submanifolds $L\subset M$ that intersect $\partial M$ transversely,
with the property that $\theta|L\in \Omega^1(L)$ is exact, and
$\theta$ vanishes to infinite order along the boundary $\partial
L=L\cap \partial M$. 

We will be considering Lagrangians in the symmetric product of $\C$,
$\Sym^k(\C)$. There is a quotient map $\pi\colon \C^k\to \Sym^k(\C)$,
and also there is a diffeomorphism $\C^k\cong\Sym^k(\C)$. 
The
relationship between Lagrangians in a symmetric product of a curve
with the ${\mathfrak S}_k$-invariant Lagrangians in the $k$-fold
Cartesian product is unclear; but there is a nice bridge offered by
work of Perutz~\cite{Perutz}, building on work of
Varouchas~\cite{Varouchas}, who constructs a new symplectic form on
the symmetric product that agrees with group-invariant the symplectic
structure on the Cartesian product on an open set. 
The case at hand is a
particularly simple, local version. (See Proposition~\ref{prop:IdentifyLagrangians} below.)

In the interest of concreteness, we find it convenient to have some explicit
parameterizations. Consider the
map from the infinite cylinder
$A=\C\setminus\{0\}\cong \R \times S^1=\R\times(\R/2\pi\Z)$ to $\C$,
specified by
\[ p(z)=\OneHalf\left(z+\frac{1}{z}\right),\]
This map has the following properties:
\begin{itemize}
\item $p$ is is a branched double-cover, with two branched points
  at $1$ and $-1$.
\item $p$ is proper
\item The image under $p$ of the circle $\{0\}\times S^1$ is the interval
  $[-1,1]\subset \C$.
\end{itemize}

Given a $k$-element subset $\x\subset S^1\setminus \{\pm
1\}$, we can view $\R\times \x$
as a subset of $\Sym^k(A)$. This image is a smooth submanifold, 
 whose image under $\Sym^k(p)$ is a smooth submanifold
of $\Sym^k(\C)$. We denote this subspace $\Lambda_{\x}$.

More generally, fix ${\mathbf x}$ as above and an element $\phi\in \R^{\geq 0}$.  
There is a submanifold of $A\cong \R\times S^1$ of elements of
the form $\{(t,e^{i\phi\cdot t/2}\x)\}_{t\in\R}$, which induces
a submanifold $\Lambda_{\x}^{\phi}\subset \Sym^k(\C)$. 
Clearly, $\Lambda_{\x}^{0}=\Lambda_{\x}$.

Let $t\colon A \to \R$ be projection to the first coordinate or, equivalently,
$z\in \C\setminus\{0\}\mapsto \log|z|$.
Consider the function $\delta\colon C^k\to \R\geq 0$ defined by
\[ \delta(z_1,\dots,z_k)=\min(\min_i t(z_i),\min_{i\neq j} |z_i-z_j|),\]
which descends to a continuous function $\Sym^k(\C)\to \R$,
so that
\[ \delta^{-1}(0)=\Delta\cup \left(S^1\times \Sym^{k-1}(A)\right).\]

The following is an adaptation of a theorem of Perutz~\cite{Perutz};
see also Varouchas~\cite{Varouchas}.

\begin{prop}
  \label{prop:IdentifyLagrangians}
  Given any bounded open set $W\subset \Sym^k(\C)$ containing
  $\Sym^k[-1,1]$ and any $\eta>0$, 
  there is a smooth plurisubharmonic
  function $\psi\colon W\to \R$ with the following properties:
  \begin{itemize}
  \item 
    Given any $k$-elements subset $\x\subset (S^1\setminus\{\pm 1\})$
    with $\delta(\x)\geq\eta$, 
    the intersection of $W$ with the submanifold
      $\Lambda_{\mathbf x}\subset \Sym^k(\C)$
      is Lagrangian
      with respect to the symplectic structure $d d^{\C}\psi$.
    \item Given $s\geq 0$, there is an exact Hamiltonian diffeomorphism 
      $\Phi^s\colon \R\times W\to W$
      with the property that
      $\Phi^{s}(\Lambda_{\mathbf x}\cap W)=\Lambda^{s}_{\mathbf x}\cap W$.
  \end{itemize}
\end{prop}

\begin{proof}
  This follows easily from Varouchas's ``Lemme
  Principal''~\cite{Varouchas}, which which we state in a slightly
  simplified form. Given
  the data:
  \begin{itemize}
  \item 
    Open subsets $U$, $V$, $W$, and $X$ in $\C^n$ 
    so that $U$, $V$, and $W$ are bounded, 
    with ${\overline
      U}\subset V$, ${\overline V}\subset W$, ${\overline W}\subset X$
  \item a continuous, strictly pluri-subharmonic function $\phi\colon X\to \R$
    so that $\phi|_{X\setminus U}$ is smooth.
  \end{itemize}
  there is a smooth, strictly pluri-subharmonic function $\chi\colon W\to \R$
  so that 
  \[ \psi|_{W\setminus (V\cap W)}=\chi|_{W\setminus (V\cap W)}.\]

  The function  $t^2\colon \C^k\to \R$ given by $t^2=\sum_{i=1}^k
  |t_i|^2$ is a smooth.  Let $X=(\C\setminus \{0\})^{\times k}$,
  and $\Pi\colon \C^{\times k}\to \Sym^k(\C)$ be the quotient map.
  As
  in~\cite{Varouchas}, since $t^2\colon A \to \R$ is a smooth,
  strictly pluri-subharmonic function and $\Pi\circ p^{\times k}\colon
  A^k\to \Sym^k(\C)$ is a branched cover, the push-forward $\psi=(\Pi\circ p^{\times k})_*(t^2)$ is a
  continuous, strictly pluri-subharmonic function on $X$.
  Given $W$ as in the statement of the proposition, apply Varouchas' lemma to $\psi$,
  $U=\delta^{-1}(0,\frac{\eta}{3})$, and $V=\delta^{-1}(0,\frac{\eta}{2})$.   The exact
  Hamiltonian is associated to the Hamiltonian function $\chi$ coming from the lemma.  Since
  $\chi$ agrees with $t^2$ over the complement of $V$, it is easy to
  see that the integral of $\chi$ preserves $\delta$ over that
  set. The second point now follows readily; see  Equation~\eqref{eq:HamiltonianFlow}.
\end{proof}

We fix the following data:
\begin{itemize}
  \item integers $m$ and $k$ with $0<k<m$
  \item $m$ basepoints 
    $-1=O_1,O_2,\dots,O_{m-1},1=O_m$
    so that there is a positively oriented arc in $S^1$
    from $O_{i+1}$ ot $O_i$ containing no other $O_j$.
  \item $m-1$ additional points $p_1,\dots,p_{m-1}$,
    so that $z_i$ is on the arc from $O_{i+1}$ to $O_i$.
  \item $\delta(z_i,z_{i+1})\geq 2/m-1$.
\end{itemize}
Choose $W\subset \Sym^k(\C)$ as in
Proposition~\ref{prop:IdentifyLagrangians}, and let $\psi$ be the
function supplied by that proposition.  There are
$\binom{m-1}{k}$ Lagrangians $\Lambda_\x\cap
  W$, associated to the $k$-element subsets of
  $\{p_1,\dots,p_{m-1}\}$.  We will be considering these as our basic
  objects in the wrapped Fukaya category of $W$.

Let $\Psi^c\colon W\to W$ be the time $\log(c)$ flow of the Liouville vector field
induced from $\psi$.
It is an easy consequence of Equation~\eqref{eq:LiouvilleFlow}
that
\begin{equation}
  \label{eq:ImageUnderLiouville}
  \Psi^c(\Lambda_\x^\phi)=
  \Lambda_\x^{\phi/c}.
\end{equation}

\subsection{The relative Fukaya category}

In our case, the symplectic manifold $W$ is equipped also with $m$
divisors, of the form $\{O_i\}\times \Sym^{k-1}(\HD)$. Correspondingly,
as in~\cite{HolDisk}; see also, cite~\cite{PerutzSheridan} for a
general construction, the Floer complexes are to be thought of as
modules over a polynomial algebra $\Field[v_1,\dots,vm]$.
Specifically, for $L_1=\Phi_H(\Lambda_{\x_1})$, $L_2=\Lambda_{\x_2}$,
the complex $\CF(L_1,L_2)$ is a module over $\Field[v_1,\dots,v_m]$
freely generated by $L_1\cap L_2$, with differential determined by
\[ \partial \x= \sum_\y \sum_{\{\phi\in\ModFlow(\x,\y)\mid
  \Mas(\phi)=1\}} \#\ModFlow(\phi) \cdot \y \cdot v_1^{n_{O_1}(\phi)}\cdots
v_m^{n_{O_m}(\phi)}.\] Here, $n_{O_i}(\phi)$ denotes the
(non-negative) algebraic intersection number of $\phi$ with the
divisor $\{O_i\}\times \Sym^{k-1}(\HD)$. The moduli space
$\ModFlow(\x,\y)$ is to be taken with respect to a suitable
perturbation of the Floer equation.

%% file: pongrev.tex
\section{Lifted partial permutations and the pong algebra}
\label{sec:LiftPerm}

We recall the construction of the pong algebra from~\cite[Section~4]{Pong};
we refer the reader to that reference for a more leisurely account.

Let $r_t\colon \R\to\R$ be the reflection $r_t(x)=2t-x$; and consider
the subgroup $G_m$ of the reflection group of the real line generated
by $r_{\OneHalf}$ and $r_{m-\OneHalf}$. The quotient of the integral lattice by
this group of rigid motions is naturally an $m-1$ point set; generated
by $\{1,\dots,m-1\}$. Let
\begin{equation}
  \label{eq:DefQ1}
  Q_1 \colon \Z \to \{1,\dots,m-1\}
\end{equation} denote this quotient map.

Note that $G_m$ also acts on the set $\OneHalf+ \Z$. The quotient of $\OneHalf
+ \Z$ by $G_m$ is naturally the $m$-point set, $\{\OneHalf,\dots,
m-\OneHalf\}$.  We think of these points as being in one-to-one
correspondence with the underlying variables in the pong algebra,
where the point $j\in \{\OneHalf,\dots,m-\OneHalf\}$ corresponds to the 
variable $v_{\OneHalf+j}$.
Explicitly, we have the map
\begin{equation}
  \label{eq:DefQ2}  Q_2\colon \Z+\OneHalf\to \{1,\dots,m\},
\end{equation}
defined so that $Q_2(j-\OneHalf)$ is the element $i\in\{1,\dots,m\}$
with
$i\equiv j\pmod{2m-2}$ or $i\equiv 2-j\pmod{2m-2}$.

A $G_m$ invariant subset $\LiftS$ of $\Z$ has a natural quotient $\LiftS/G_m$,
which is a subset of $\{1,\dots,m-1\}$. 

\begin{defn}
  A {\em lifted partial permutation on $k$ letters} 
  is a pair $({\widetilde S}, {\widetilde f})$ where:
  \begin{itemize}
  \item ${\widetilde S}\subset \Z$
    is a $G_m$-invariant subset 
  \item ${\widetilde f}\colon {\widetilde S} \to \Z$
    is a $G_m$-equivariant map;
  \end{itemize}
  subject to the following two conditions:
  \begin{itemize}
    \item  ${\widetilde S}/G_m$ consists of $k$ elements
    \item   the induced map ${\widetilde f}\colon {\widetilde S}/G_m\to \Z/G_m$
      is injective.
  \end{itemize}
\end{defn}

\begin{defn}
  A lifted partial permutation $({\widetilde S},{\widetilde f})$ has a
  {\em weight vector} ${\vec
    \weight}=(\weight_1,\dots,\weight_m)\in(\OneHalf \Z)^m$, specified
  by
  \[ \weight_j({\widetilde f})= \OneHalf 
  \#\{i\in{\widetilde S}\big| i<j-\OneHalf<{\widetilde f}(i)~\text{or}~
  i>j-\OneHalf>{\widetilde f}(i)\}.
  \]
\end{defn}

We extend the weight vector to $\Field[v_1,\dots,v_m]$ so that $\weight(v_i)$
is the $i^{th}$ basis vector in $\Z^m$.

\begin{defn}
  A {\em crossing} in a lifted partial permutation ${\widetilde f}$ is
  an equivalence class of pairs of integers $(i,j)$ with the property
  that $i<j$ and ${\widetilde f}(i)>{\widetilde f}(j)$. We say that
  $(i,j)$ and $(i',j')$ determine the same crossing if there is some
  $g\in G_m$ so that $\{g\cdot i, g\cdot j\}=\{i',j'\}$.  We write
  $\langle i,j\rangle$ for the equivalence class of the pair of
  integers$(i,j)$.
  Let $\Cross({\widetilde f})$ denote the set of crossings in ${\widetilde f}$.
\end{defn}

Note that $\langle i,j\rangle \in \Cross({\widetilde f})$ does not exclude
cases where $[i]=[j]$.
Let $\cross({\widetilde f})$ denote the number of
crossings in ${\widetilde f}$.

Let 
$({\widetilde f},{\widetilde S})$ and 
$({\widetilde g},{\widetilde T})$ be two partial permutations with 
${\widetilde T}={\widetilde f}({\widetilde S})$.
Then, the composite $({\widetilde g}\circ {\widetilde f},{\widetilde S})$
is a lifted partial permutation.

It is elementary to verify that 
\begin{align*}
  \weight({\widetilde g}\circ {\widetilde f})&\leq 
  \weight({\widetilde g}) + \weight({\widetilde f}) \\
  \cross({\widetilde g}\circ {\widetilde f})&\leq 
  \cross({\widetilde g}) + \cross({\widetilde f})
\end{align*}

The pong algebra $\Pong{m}{k}$ is the algebra over
$\Field[v_1,\dots,v_m]$ freely generated by lifted partial
permutations, with a multiplication map
\[ \mu_2\colon \Pong{m}{k}\otimes_{\Field[v_1,\dots,v_m]} \Pong{m}{k}
\to \Pong{m}{k} \]
characterized by
\[ \mu_2([{\widetilde
  f},{\widetilde S}], [{\widetilde g},{\widetilde T}])]
=\begin{cases}
0 & \text{if ${\widetilde T}\neq {\widetilde f}({\widetilde S)}$} \\
0 & \text{if $\cross({\widetilde g}\circ {\widetilde f})<\cross({\widetilde
    g}) + \cross({\widetilde f})$} \\
v\cdot [{\widetilde g}\circ {\widetilde f},{\widetilde S}] &{\text{otherwise,}} 
\end{cases}\]
where $v$ is the monomial in $v_1,\dots,v_m$ chosen so that
\[ 
\weight(\mu_2([{\widetilde f},{\widetilde S}], 
[{\widetilde g},{\widetilde T}]))
= \weight[{\widetilde f},{\widetilde S}]+
\weight[{\widetilde g},{\widetilde T}].\]

Given $a,b\in\Pong{m}{k}$, we abbreviate $\mu_2(a,b)$ by $a\cdot b$.

For each $\langle i,j\rangle\in\Cross({\widetilde f})$, there is a new lifted partial permutation ${\widetilde f}_{\langle i,j\rangle}$ characterized as follows:
\[
{\widetilde f}_{\langle i,j\rangle}(k)=
\begin{cases}
{\widetilde f}(k) & {\text{if~$[k]\not\in\{[i],[j]\}$}} \\
g\cdot {\widetilde f}(j) &{\text{if $k=g\cdot i$}} \\
g\cdot {\widetilde f}(i) &{\text{if $k=g\cdot j$}} 
\end{cases}.\] 

It is elementary to verify that 
\begin{align*}
  \weight({\widetilde f}_{\langle i,j\rangle})&\leq \weight({\widetilde f}) \\
  \cross({\widetilde f}_{\langle i,j\rangle})&\leq  \cross({\widetilde f})-1
\end{align*}

Given $\langle i,j\rangle\in\Cross({\widetilde f},{\widetilde S})$, 
let $\partial_{\langle i,j\rangle}{\widetilde f}\in\Pong{m}{k}$ be the element defined by
\[ \partial_{\langle i,j\rangle}{\widetilde f}= 
\begin{cases} 0 & {\text
    {if $\cross({\widetilde f}_{\langle i,j\rangle})<
      \cross({\widetilde f})-1$}} \\
  v \cdot {\widetilde f}_{\langle i,j\rangle} & {\text{otherwise,}}
\end{cases} \]
where now $v$ is the monomial in $v_1,\dots,v_m$ characterized by the property that
\[ 
\weight[{\widetilde f},{\widetilde S}]
= \weight[{\widetilde f}_{\langle i,j\rangle},{\widetilde S}].
\]
Define a map
\[ \partial\colon \Pong{m}{k}\to \Pong{m}{k}, \]
characterized by
\[ \partial({\widetilde f},{\widetilde S}) =\sum_{\langle i,j\rangle\in\Cross({\widetilde f},{\widetilde S})} \partial_{\langle i,j\rangle} [{\widetilde f},{\widetilde S}].\]

With the above definitions, $\Pong{m}{k}$ is a differential graded
algebra over $\Field[v_1,\dots,v_m]$;
see~\cite[Proposition~\ref{P:prop:PongIsAlg}]{Pong} for details.

It will be convenient to have the following:

\begin{lemma}
  \label{lem:NoOuterUs}
  If
  $\cross({\widetilde f}\circ{\widetilde g})=
  \cross({\widetilde f})+\cross({\widetilde g})$
  then 
  \[ \weight_1({\widetilde f}\circ{\widetilde g})=
  \weight_1({\widetilde f})+\weight_1({\widetilde g}) 
  \qquad{\text{and}}\qquad
  \weight_m({\widetilde f}\circ{\widetilde g})=
  \weight_m({\widetilde f})+\weight_m({\widetilde g}). \]
\end{lemma}

\begin{proof}
  If $\weight_1({\widetilde f}\circ {\widetilde g})\neq \weight_1({\widetilde f}))+\weight_1({\widetilde g})$, then there exists some $1<i$ with
  \begin{equation}
    \label{eq:cr1}
    1<i \text{~such that~} {\widetilde g}(i)<1; 
  \end{equation}
  and 
  \begin{equation}
    \label{eq:cr2}
    {\widetilde f}\circ {\widetilde g}(i)>1.
  \end{equation}
  On the other hand, Equation~\eqref{eq:cr1} implies that
  \[ {\widetilde g}(1-i)=1-{\widetilde g}(i)>1-i;\]
  while Equation~\eqref{eq:cr2} implies that
  \[ {\widetilde f}\circ{\widetilde g}(1-i)>{\widetilde g}(1-i);\]
  i.e. $i$ and $1-i$ have a crossing in ${\widetilde g}$, while
  ${\widetilde g}(i)$ and ${\widetilde g}(1-i)$ have a crossing in
  ${\widetilde f}$.

  An analogous argument works for $\weight_m$.
\end{proof}

%% file: heegpong.tex
\section{The wrapped diagram}
\label{sec:HeegPong}

Our aim here is to give a particularly convenient description of the wrapped
diagram for $\Sym^k(M)$, equipped with $\Lambda_\x$.

Consider the plane  $\R^2$, decorated with the following data:
\begin{itemize}
  \item  an infinite grid of
    vertical lines 
    ${\widetilde\alphas}=\{{\widetilde \alpha}_i=i\times \R\}_{i\in\Z}$;
  \item horizontal lines 
    ${\widetilde\betas}=\{{\widetilde \beta}_i=\R\times i\}_{i\in\Z}$;
  \item an infinite set of punctures at the points $\{(\OneHalf +
    i,\OneHalf +i)\}_{i\in \Z}$, so that $(\OneHalf+i,\OneHalf+i)$ is 
    labeled by $O_j$, where $j=Q_2(\OneHalf+i)$, with $Q_2$ as in Equation~\eqref{eq:DefQ2}.
    \end{itemize}

    The symmetry group of this picture is generated by the two
    $180^\circ$-rotations with fixed point at $(\OneHalf,\OneHalf)$,
    and the one with fixed point at $(m-\OneHalf,m-\OneHalf)$. Let
    ${\mathbb G}_m$ denote this group of rigid motions. Note that ${\mathbb G}_m\cong G_m$, induced by the
    restriction of ${\mathbb G}_m$ to the diagonal line in $\R^2$.

The quotient space $\R^2/{\mathbb G}_m$ is homeomorphic to the disk $\HD$ with
two order $2$ orbifold points, which are the points marked $O_1$ and
$O_m$. There are an addition $m-2$ marked points, labeled $O_i$ for
$i=2,\dots,m-1$. 

The vertical lines $\Z\times \R=\{{\widetilde \alpha}_i\}_{i\in\Z}$
project to $m-1$ embedded lines
$\{\alpha_i\}_{i=1}^{m-1}$ in $\HD$. Similarly, the horizontal lines
$\{{\widetilde\beta}_i\}_{i\in\Z}$, project to $m-1$
embedded lines $\{\beta_i\}_{i=1}^{m-1}$ in $\HD$. We label the lines in
$\HD$ so that $\alpha_i$ is the image of ${\widetilde \alpha}_j$, for
any $j\in\Z$ with $Q_1(j)=i$, with $Q_1$ as in Equation~\eqref{eq:DefQ1};
similarly, $\beta_i$ is the image of ${\widetilde \beta}_j$.

\begin{lemma}
  \label{lem:WrapDiagram}
  Consider $\HD=\R^2/{\mathbb G}_m$, equipped with $\{\alpha_i\}_{i=1}^{m-1}$,
  $\{\beta_i\}_{i=1}^{m-1}$, and the markings $\{O_i\}_{i=1}^m$. 
  This is a diagram for the wrapping of $\{\alpha_i\}_{i=1}^{m-1}$.
\end{lemma}

\begin{proof}
  What we mean is the following.  Consider $\HD$ as above, equipped
  with the vertical circles ${\widetilde\alphas}$. ${\mathbf G}_m$
  contains an index two subgroup of translations, generated by
  $(x,y)\mapsto(x+2m,y+2m)$.  Consider the cylinder $A$ obtained as
  the quotient
  \begin{align}
    A=\frac{\R\times \R}{(2m,2m)\cdot \Z}.
  \end{align}
  Moreover, there is a branched covering map from $A$ to $\HD$,
  with branching at $O_1$ and $O_m$.
  
  Equip $\Sym^k(\HD)$ with the symplectic structure from
  Proposition~\ref{prop:IdentifyLagrangians}, chosen so that the images
  in $\Sym^k(\HD)$ of the 
  manifolds $\alpha_{x_1}\times\dots\times \alpha_{x_k}\subset
  \HD^{\times k}$ for all subsequences $\x\subset \{1,\dots,m-1\}$ are
  Lagrangian.  Note that our explicit parametrizations here differ
  from the ones described around
  Proposition~\ref{prop:IdentifyLagrangians} by a linear
  transformation. In particular, these manifolds are equivalent to the
  submanifolds $\Lambda_\x$ from that proposition. Moreover, 
  the Liouville flow carries
  $\alphas_{x_1}\times\dots\times\alpha_{x_k}$ to
  $\beta_{x_1}\times\dots\times\beta_{x_k}$.  
  after some (positve) time.
\end{proof}

See Figures~\ref{fig:QuotDiag} and~\ref{fig:QuotDiag2} for examples.

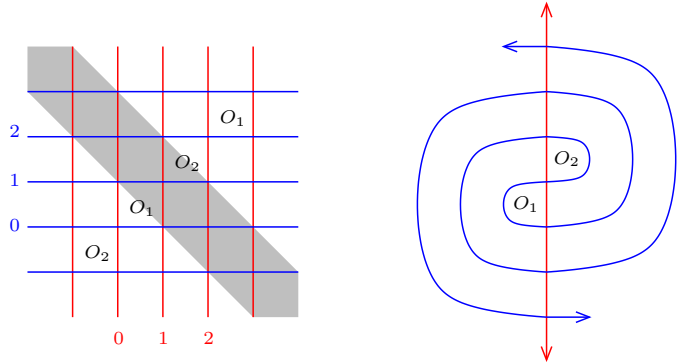
\begin{figure}[ht]
\input{QuotDiag.pstex_t}
\caption{\label{fig:QuotDiag} {\bf{Heegaard diagrams.}}
The quotient of the infinite grid diagram on the left is the diagram (with $m=2$) on the right. On the left, we have shaded a fundamental domain for the
${\mathbb G}_2$ action.}
\end{figure}

\begin{figure}[ht]
\input{QuotDiag2.pstex_t}
\caption{\label{fig:QuotDiag2} {\bf{Heegaard diagrams.}}
The case where $m=3$.}
\end{figure}
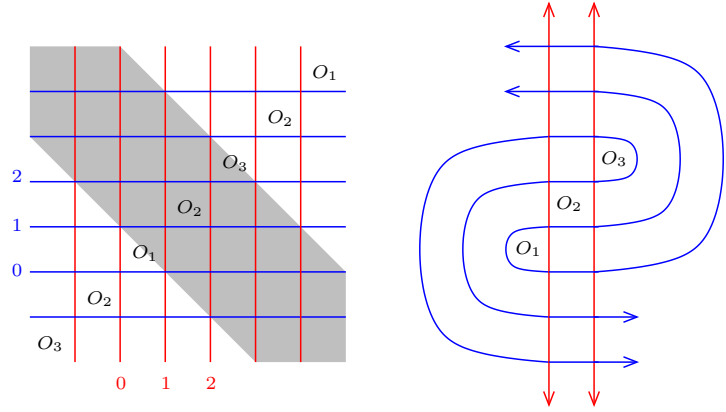

As usual, a {\em $k$-fold Heegaard state} is a $k$-tuple of points (for some $0<k<m$), with the property that each point lies on $\alpha_i\cap \beta_j$,
no two points lie on the same $\alpha_i$, and no two points lie on the
same $\beta_j$.

Given a lifted partial permutation ${\widetilde f}$, we can form its
graph 
\[ \Gamma_{\widetilde f}=\{(i,{\widetilde f}(i))\mid i\in {\widetilde S}\}\subset \R^2.\]
Clearly $\Gamma_{\widetilde f}$ is invariant under ${\mathbb G}_m$;
as such we can form the associated subset $\x({\widetilde f})\subset \HD$.

\begin{lemma}
  The above map sets up a one-to-one correspondence between ($k$-element) lifted partial permutations and ($k$-fold) Heegaard states for $\HD$.
\end{lemma}

\begin{proof}
  The proof is straightforward.
\end{proof}

\begin{defn}
  \label{def:doms2}
  Suppose that $\HD$ is a surface, equipped with two sets of curves
  $\alphas=\{\alpha_i\}$ and $\betas=\{\beta_i\}$. We think of the curves as
  giving $\HD$ a $CW$-complex structure, with $0$-cells the
  intersection points between $\alpha_i$ and $\beta_j$, $1$-cells the arcs
  in $\alpha_i$ and $\beta_j$, and two-chains the components of
  $\HD\setminus (\alphas\cup\betas)$. Thus, a two-chain can be thought of
  as an assignment of integers to each component of
  $\HD\setminus(\alphas\cup\betas)$. Fix some intersection point $x$ of
  $\alpha_i$ with $\beta_j$.  We say that $x$ is a {\em{corner}} if
  the local multiplicities $A$, $B$, $C$, and $D$ around $x$, as
  pictured in Figure~\ref{fig:Cornerless}, satisfy
  $A+D\neq B+C$. In fact, we
  say that $x$ is an {\em initial $(\alpha,\beta)$ corner} if $B+C=A+D+1$;
  if $B+C=A+D-1$, we say $x$ is a {\em terminal $(\alpha,\beta)$-corner}.
  A domain is a {\em cornerless domain} if it has no corner.
\end{defn}

\begin{figure}[ht]
\input{Cornerless.pstex_t}
\caption{\label{fig:Cornerless} {\bf{Corner conventions.}}
}
\end{figure}
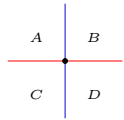

The space of cornerless domains is an abelian group.  If $\x$ and $\y$
are Heegaard states, let $\doms(\x,\y)$ denote the space of domains
with initial corner at the components of $\x\setminus(\x\cap\y)$ and terminal corner at
the components of $\y\setminus(\x\cap\y)$. 
(Algebraically, $\doms(\x,\y)$ is an affine
space for the space of cornerless domains.) As explained
in~\cite{HolDisk}, for $k\geq 3$, 
$\doms(\x,\y)$ is identified with a space of relative
homotopy classes of Whitney disks connecting $\x$ to $\y$, denoted
there $\pi_2(\x,\y)$.

\begin{lemma}
  \label{lem:UniqueHomotopyClass}
  Given Heegaard states $\x$ and $\y$, corresponding to lifted partial permutations
  $({\widetilde f},{\widetilde S})$ and
  $({\widetilde g},{\widetilde T})$,
  there is a $\phi\in\doms(\x,\y)$
  with compact support if and only if
  ${\widetilde S}={\widetilde T}$ and
  ${\widetilde f}({\widetilde S})={\widetilde g}({\widetilde T})$.
  Moreover, if $\phi$ exists, then it is unique.
\end{lemma}
\begin{proof}
  Given $\x$ and $\y$, consider the corresponding lift to $\R^2$.  The
  hypothesis that ${\widetilde S}={\widetilde T}$ is equivalent to the
  condition that we can connect $\x$ to $\y$ by a path $A$ in
  $\alphas=\{\alpha_i\}_{i=1}^{m-1}$.  (When the path exists,
  its uniqueness, as a relative one-chain in $\HD$, is obvious.) 
  Similarly, the condition that ${\widetilde
    f}({\widetilde S})={\widetilde f}({\widetilde T})$ is equivalent
  to the condition that we can connect ${\widetilde y}$ to
  ${\widetilde \x}$ (uniquely) by a path $B$ inside $\betas$. By
  construction, $\partial A = \partial B$, so $A-B=\partial D$, for
  some two-chain. Uniqueness follows from contractability of $\HD$.
\end{proof}

Suppose that $({\widetilde f},{\widetilde S})$ is a lifted partial
permutation with graph ${\widetilde \x}$ and Heegaard state $\x$.  Let
$A$ be the set of vertical lines in $\R^2$ that connect ${\widetilde x}$
to the diagonal line.  Given $1\leq i\leq m$, the weight of 
${\widetilde S}$ at $i$ can be
interpreted as the number of times $A$ crosses the horizontal line 
$\R\times (i-\OneHalf)$. 

\begin{lemma}
  \label{lem:IdentifyWeights}
  Suppose that $({\widetilde f},{\widetilde S})$ and 
  $({\widetilde g},{\widetilde T})$ are lifted
  partial permutations with corresponding Heegaard states $\x$ and
  $\y$, which can be connected by some $\phi\in \doms(\x,\y)$. Then,
  the local multiplicity of $\phi$ at $O_i$ coincides with
  $\weight(\widetilde f)-\weight(\widetilde g)$.
\end{lemma}

\begin{proof}
  Let $\x_0$ denote the Heegaard states corresponding to the identity
  map on ${\widetilde S}$.  Let $A$ be (oriented) vertical path from
  $\x$ to $\x_0$, so that $\weight(\x)$ counts half the number of
  times $A$ crosses $\R\times (i-\OneHalf)$. We can think of $\R\times
  (i-\OneHalf)$ as the union of two rays ${\widetilde r}_i$ and
  ${\widetilde r}_i'$ starting at $(i-\OneHalf,i-\OneHalf)$. Then,
  $\weight(\x)$ is one half the oriented intersection number of $r_i$
  with $A$ plus the oriented intersection number of $r_i'$ with $A$.

  Let $r_i$ resp. $r_i'$ denote the image in ${\mathbb H}$ of
  ${\widetilde r}_i$ resp. ${\widetilde r}'_i$. Observe that $r_i$ and
  $r_i'$ are paths in ${\mathbb H}$ from $O_i$ to infinity that avoid
  $\betas$. (Indeed, for $i=1$ and $m$, $r_i=r_i'$.) Thus,
  $\weight(\x)-\weight(\y)$ is the algebraic intersection number of
  $r_i$ with $\partial \phi$, which in turn coincides with the winding
  number of $\partial \phi$ around $O_i$, and hence the local
  multiplicity of $\phi$ at $O_i$. Since the same remarks apply for
  $r_i'$, the result follows.
\end{proof}

\begin{figure}[ht]
\input{Bigon.pstex_t}
\caption{\label{fig:Bigon} {\bf{A lifted bigon.}}  Consider $m=3$ and
  $k=1$. The black dot corresponds to the lifted partial permutation
  sending $2$ to $-1$, while the white dots corresponds to the map
  sending $2$ to $2$.  There is a (shaded) bigon from the black to the white
  dot with multiplicity $1$ at $O_1$ and $O_2$. The ray $r_2$ is indicated in the picture.}
\end{figure}
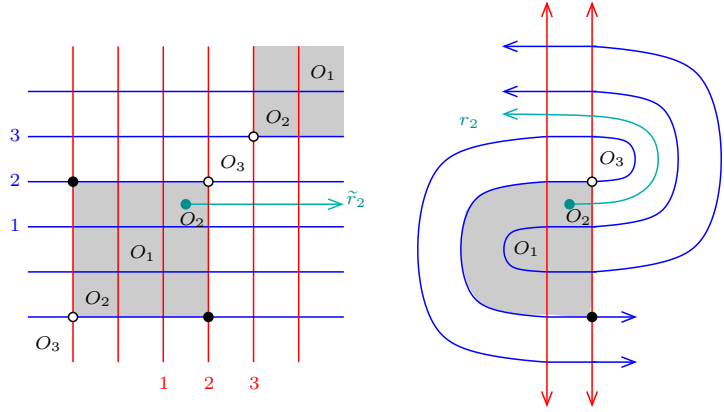

Under this correspondence, $\Cross({\widetilde f})$ corresponds
${\mathbb G}_m$ orbits of pairs of points in $\x$ of the form
$(x_1,y_1),(x_2,y_2)$ so that $x_1<x_2$ and $y_1>y_2$.

Thus, for each crossing, there is a unique ${\mathbb G}_m$-orbit of
embedded rectangle $r$ in $\R^2$, whose upper left corner is at $(x_1,y_1)$,
and whose lower right corner is at $(x_2,y_2)$. The graph of
${\widetilde f}_{\langle i,j\rangle}$ is the subset of
$\Z\times \Z$ obtained from ${\widetilde\x}$ (the graph of ${\widetilde f}$)
by removing the ${\mathbb
  G}_m$-orbits of $(x_1,y_1)$ and $(x_2,y_2)$, and replacing them with
the ${\mathbb G}_m$ orbits of $(x_1,y_2)$ and $(x_2,y_1)$. 

The condition that $\cross({\widetilde f}_{\langle
  i,j\rangle})=\cross({\widetilde f})-1$ is equivalent to the
condition that $r$ is {\em empty}: i.e. it does not contain any points
in $\x$ in its interior. (This is equivalent to the condition that the
${\mathbb G}_m$ translates of $r$ are a collection of disjoint rectangles,
none of which contains a component of ${\widetilde x}$ in its interior.)

\begin{lemma}
  \label{lem:IdentifyDifferential}
  Given a lifted partial permutation ${\widetilde f}$ and graph
  ${\widetilde\x}$, there is a one-to-one
  correspondence between the resolutions of the crossings in
  ${\widetilde f}$ and ${\mathbb G}_m$-orbits of embedded rectangles
  in $\R^2$ whose upper left and lower right corners are on ${\widetilde x}$.
  Moreover, the following conditions are equivalent:
  \begin{enumerate}
    \item Each (or any) rectangle in $\R^2$ in the ${\mathbb G}_m$ orbit 
      is empty.
    \item The image of the rectangle in $\HD$ is an empty 
      bigon (in the case where ${\mathbb G}_m$-orbit of the rectangle has isotropy group ${\mathbb Z}/2{\mathbb Z}$) or an empty rectangle
      (when the isotropy group is trivial).
    \item 
      The number of crossings in the resolution is one less than the number of 
      crossings in ${\widetilde f}$.
    \end{enumerate}
\end{lemma}

\begin{proof}
  It was already noted that the crossings in ${\widetilde f}$
  correspond to embedded rectangles in ${\mathbb R}^2$, from
  ${\widetilde f}$ to ${\widetilde g}$. The rectangle is empty
  precisely when $\cross({\widetilde g})=\cross({\widetilde f})-1$.
  The isotropy group of the ${\widetilde G}_m$-orbit of any rectangle
  is either trivial or $\Zmod{2}$. The orbits of empty rectangles project
  to either empty rectangles or empty bigons in $\HD$.
  Note that $\partial r$ is embedded precisely when the interior of $r$
  does not contain any points that are equivalent to the corners of $r$;
  see Figure~\ref{fig:Nonemb} for an example where $\partial r$ is not embedded.
\end{proof}

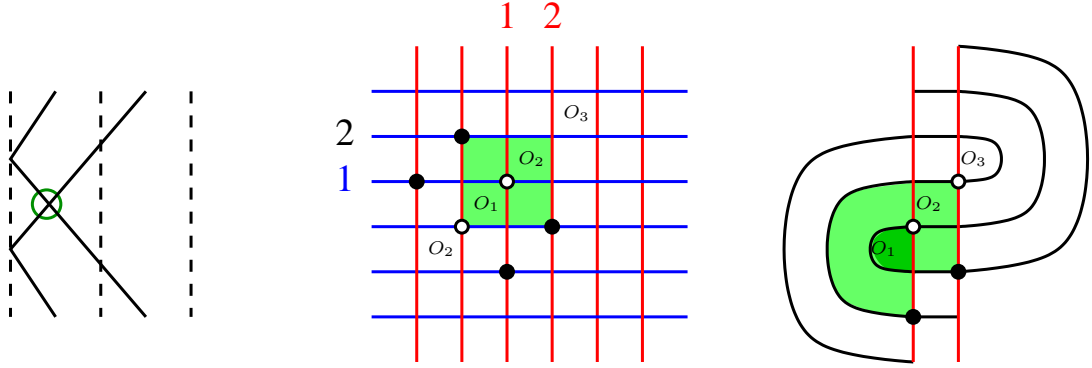
\begin{figure}[ht]
\input{Nonemb.pstex_t}
\caption{\label{fig:Nonemb} {\bf{Projection of a nonempty rectangle.}}
We have a non-empty rectangle on the left
which projects to the domain in $\HD$ pictured on the right.}
\end{figure}

The Heegaard diagram $\HD_{m,k}$ is {\em nice} in the sense of Sarkar;
thus, by~\cite{SarkarWang}, the differentials in $\CFm(\HD_{m,k})$
count empty bigons and rectangles. Thus, we could use this and
Lemma~\ref{lem:IdentifyDifferential} to deduce an identification of
chain complexes $\CFm(\HD_{m,k})\cong\Pong{m}{k}$. Instead, we invest
a little more work in understanding the combinatorics of $\HD$ to
give an alternative argument.

There is a partial ordering on Heegaard states: we write $\x\geq \y$
if there is a $\phi\in\pi_2(\x,\y)$ all of whose local multiplicities
are non-negative; with strict inquality $\x>\y$ if $\phi$ has positive
local multiplicity somewhere.

\begin{lemma}
  \label{lem:PositiveDomains}
  If $({\widetilde f},{\widetilde S})$ and $({\widetilde
    g},{\widetilde T})$ are two lifted partial permutations, and $\x$
  and $\y$ be their corresponding Heegaard states.  
  The following conditions are equivalent:
  \begin{enumerate}[label=(P-\arabic*),ref=(P-\arabic*)]
  \item \label{P:Greater}
    $\x\geq \y$
  \item 
    \label{P:CountLines}
    For each $(i,j)$,
    \[ \#\{a\in \Z\mid a<i, {\widetilde g}(a) <j<{\widetilde f}(a)\}
    \geq \#\{a\in \Z\mid a<i, {\widetilde g}(a) >j>{\widetilde f}(a)\}.\]
  \item 
    \label{P:ResolveCrossings}
    There is a sequence of crossings in ${\widetilde f}$ with
    the property that ${\widetilde g}$ is obtained from ${\widetilde
      f}$ by resolving those crossings.
  \end{enumerate}
  Moreover, if $\x\geq \y$, then 
  \begin{equation}
  \Mas(\x)-\Mas(\y)=\#\cross({\widetilde f})-\#\cross({\widetilde g});
  \label{eq:MaslovDifference}
  \end{equation}
\end{lemma}

\begin{proof}
  As in the proof of Lemma~\ref{lem:IdentifyWeights}, we can interpret
  \[ \#\{a\in \Z\mid a<i, {\widetilde g}(a) <i<{\widetilde f}(a)\}-\#\{a\in
  \Z\mid {\widetilde g}(a) >j>{\widetilde f}(a)\} \]
  as the local
  multiplicity of $\phi$ at the point $(i+\OneHalf,j+\OneHalf)/{\mathbb G}_m$.
  It follows at once that Properties~\ref{P:Greater} and~\ref{P:CountLines} are equivalent.

  Property~\ref{P:ResolveCrossings} clearly implies Property~\ref{P:CountLines}.
  To see that~\ref{P:CountLines}$\Rightarrow$\ref{P:ResolveCrossings}
  we argue as follows.
  When $\x=\y$, the result is obvious. Suppose $\x> \y$, 
  then there is some $i_1<i_2$ so that:
  \[ {\widetilde
    f}(i_1)>j>{\widetilde g}(i_2)\qquad{\text{and}}\qquad
  {\widetilde f}(i_1)<j.\] 
  (In particular, $\langle i_1,i_2\rangle$ is a
  crossing in ${\widetilde f}$.) Choose $i_2>i_1$ minimal with this
  property, and 
  let ${\widetilde x}'$ correspond to ${\widetilde
    f}_{\langle i_1,i_2\rangle}$. Minimality of $i_2$ ensures that 
  $\x>\x'$ and $\x'\geq \y$. 
  It is easy to see that $\x'\geq \y$. 
  Indeed, the sequence is constructed 
  Since $\cross(\x')=\cross(\x)-1$, this process must terminate after at
  most $\cross(\x)$ steps.

  If $\x'$ is obtained from $\x$ by resolving single crossing, and
  $\phi\in\pi_2(\x,\x')$ is the corresponding domain in $\HD$, then
  $\Mas(\phi)=1$. This is true because $\phi$ is a bigon or a
  rectangle (Lemma~\ref{lem:IdentifyDifferential});
  both of these are easily seen to have Maslov index one.
  The Maslov index is additive under jutapositions,
  and the algorithm
  establishing~\ref{P:CountLines}$\Rightarrow$\ref{P:ResolveCrossings}
  gave a sequence $\{{\widetilde f}_i\}_{i=1}^n$ with ${\widetilde
    f}_1={\widetilde f}$, ${\widetilde f}_n={\widetilde g}$, and
  $\Mas(\x_i)-\Mas(\x_{i+1})=1$, $\cross({\widetilde
    f}_{i})-\cross({\widetilde f}_{i+1})=1$.
  Equation~\eqref{eq:MaslovDifference} follows.
\end{proof}

\begin{prop}
  \label{prop:IdentifyComplexes}
  There is an isomorphism of chain complexes
  $\CFm(\HD_{m,k})\cong \Pong{m}{k}$.
\end{prop}

\begin{proof}
  By definition, the differential on $\CFm(\HD_{m,k})$ counts
  $\phi\in\pi_2(\x,\y)$ with $\phi\geq 0$ and $\Mas(\phi)=1$.
  Equation~\eqref{eq:MaslovDifference} identifies these with the
  ${\mathbf G}_m$-orbits of empty rectangles in $\C$ which in turn, by
  Lemma~\ref{lem:IdentifyDifferential} identifies such rectangles with
  crossings in the diagram for $\x$ whose resolution drops the
  crossing number by exactly one. Lemma~\ref{lem:IdentifyWeights}
  identifies the coefficients in $\partial^-$ for $\CFm$ with the coefficients
  in $\partial$ for $\Pong{m}{k}$.
\end{proof}

%% file: QuotDiag.pstex_t
\begin{picture}(0,0)%
\includegraphics{QuotDiag.pstex}%
\end{picture}%
\setlength{\unitlength}{1243sp}%
\begingroup\makeatletter\ifx\SetFigFont\undefined%
\gdef\SetFigFont#1#2#3#4#5{%
  \reset@font\fontsize{#1}{#2pt}%
  \fontfamily{#3}\fontseries{#4}\fontshape{#5}%
  \selectfont}%
\fi\endgroup%
\begin{picture}(13333,7266)(-824,-5944)
\put(2431,-2041){\makebox(0,0)[lb]{\smash{{\SetFigFont{7}{8.4}{\rmdefault}{\mddefault}{\updefault}{\color[rgb]{0,0,0}$O_2$}%
}}}}
\put(1531,-2941){\makebox(0,0)[lb]{\smash{{\SetFigFont{7}{8.4}{\rmdefault}{\mddefault}{\updefault}{\color[rgb]{0,0,0}$O_1$}%
}}}}
\put(631,-3841){\makebox(0,0)[lb]{\smash{{\SetFigFont{7}{8.4}{\rmdefault}{\mddefault}{\updefault}{\color[rgb]{0,0,0}$O_2$}%
}}}}
\put(3331,-1141){\makebox(0,0)[lb]{\smash{{\SetFigFont{7}{8.4}{\rmdefault}{\mddefault}{\updefault}{\color[rgb]{0,0,0}$O_1$}%
}}}}
\put(-809,-2401){\makebox(0,0)[lb]{\smash{{\SetFigFont{7}{8.4}{\rmdefault}{\mddefault}{\updefault}{\color[rgb]{0,0,1}1}%
}}}}
\put(-809,-1411){\makebox(0,0)[lb]{\smash{{\SetFigFont{7}{8.4}{\rmdefault}{\mddefault}{\updefault}{\color[rgb]{0,0,1}2}%
}}}}
\put(-809,-3301){\makebox(0,0)[lb]{\smash{{\SetFigFont{7}{8.4}{\rmdefault}{\mddefault}{\updefault}{\color[rgb]{0,0,1}0}%
}}}}
\put(2161,-5551){\makebox(0,0)[lb]{\smash{{\SetFigFont{7}{8.4}{\rmdefault}{\mddefault}{\updefault}{\color[rgb]{1,0,0}1}%
}}}}
\put(1261,-5551){\makebox(0,0)[lb]{\smash{{\SetFigFont{7}{8.4}{\rmdefault}{\mddefault}{\updefault}{\color[rgb]{1,0,0}0}%
}}}}
\put(3061,-5551){\makebox(0,0)[lb]{\smash{{\SetFigFont{7}{8.4}{\rmdefault}{\mddefault}{\updefault}{\color[rgb]{1,0,0}2}%
}}}}
\put(9181,-2851){\makebox(0,0)[lb]{\smash{{\SetFigFont{7}{8.4}{\rmdefault}{\mddefault}{\updefault}{\color[rgb]{0,0,0}$O_1$}%
}}}}
\put(9991,-1951){\makebox(0,0)[lb]{\smash{{\SetFigFont{7}{8.4}{\rmdefault}{\mddefault}{\updefault}{\color[rgb]{0,0,0}$O_2$}%
}}}}
\end{picture}%

%% file: QuotDiag2.pstex_t
\begin{picture}(0,0)%
\includegraphics{QuotDiag2.pstex}%
\end{picture}%
\setlength{\unitlength}{1243sp}%
\begingroup\makeatletter\ifx\SetFigFont\undefined%
\gdef\SetFigFont#1#2#3#4#5{%
  \reset@font\fontsize{#1}{#2pt}%
  \fontfamily{#3}\fontseries{#4}\fontshape{#5}%
  \selectfont}%
\fi\endgroup%
\begin{picture}(14233,8166)(-824,-5944)
\put(2431,-2041){\makebox(0,0)[lb]{\smash{{\SetFigFont{7}{8.4}{\rmdefault}{\mddefault}{\updefault}{\color[rgb]{0,0,0}$O_2$}%
}}}}
\put(1531,-2941){\makebox(0,0)[lb]{\smash{{\SetFigFont{7}{8.4}{\rmdefault}{\mddefault}{\updefault}{\color[rgb]{0,0,0}$O_1$}%
}}}}
\put(631,-3841){\makebox(0,0)[lb]{\smash{{\SetFigFont{7}{8.4}{\rmdefault}{\mddefault}{\updefault}{\color[rgb]{0,0,0}$O_2$}%
}}}}
\put(3331,-1141){\makebox(0,0)[lb]{\smash{{\SetFigFont{7}{8.4}{\rmdefault}{\mddefault}{\updefault}{\color[rgb]{0,0,0}$O_3$}%
}}}}
\put(-809,-2401){\makebox(0,0)[lb]{\smash{{\SetFigFont{7}{8.4}{\rmdefault}{\mddefault}{\updefault}{\color[rgb]{0,0,1}1}%
}}}}
\put(-809,-1411){\makebox(0,0)[lb]{\smash{{\SetFigFont{7}{8.4}{\rmdefault}{\mddefault}{\updefault}{\color[rgb]{0,0,1}2}%
}}}}
\put(-809,-3301){\makebox(0,0)[lb]{\smash{{\SetFigFont{7}{8.4}{\rmdefault}{\mddefault}{\updefault}{\color[rgb]{0,0,1}0}%
}}}}
\put(2161,-5551){\makebox(0,0)[lb]{\smash{{\SetFigFont{7}{8.4}{\rmdefault}{\mddefault}{\updefault}{\color[rgb]{1,0,0}1}%
}}}}
\put(1261,-5551){\makebox(0,0)[lb]{\smash{{\SetFigFont{7}{8.4}{\rmdefault}{\mddefault}{\updefault}{\color[rgb]{1,0,0}0}%
}}}}
\put(3061,-5551){\makebox(0,0)[lb]{\smash{{\SetFigFont{7}{8.4}{\rmdefault}{\mddefault}{\updefault}{\color[rgb]{1,0,0}2}%
}}}}
\put(9181,-2851){\makebox(0,0)[lb]{\smash{{\SetFigFont{7}{8.4}{\rmdefault}{\mddefault}{\updefault}{\color[rgb]{0,0,0}$O_1$}%
}}}}
\put(4231,-241){\makebox(0,0)[lb]{\smash{{\SetFigFont{7}{8.4}{\rmdefault}{\mddefault}{\updefault}{\color[rgb]{0,0,0}$O_2$}%
}}}}
\put(-359,-4741){\makebox(0,0)[lb]{\smash{{\SetFigFont{7}{8.4}{\rmdefault}{\mddefault}{\updefault}{\color[rgb]{0,0,0}$O_3$}%
}}}}
\put(5131,659){\makebox(0,0)[lb]{\smash{{\SetFigFont{7}{8.4}{\rmdefault}{\mddefault}{\updefault}{\color[rgb]{0,0,0}$O_1$}%
}}}}
\put(10891,-1051){\makebox(0,0)[lb]{\smash{{\SetFigFont{7}{8.4}{\rmdefault}{\mddefault}{\updefault}{\color[rgb]{0,0,0}$O_3$}%
}}}}
\put(9991,-1951){\makebox(0,0)[lb]{\smash{{\SetFigFont{7}{8.4}{\rmdefault}{\mddefault}{\updefault}{\color[rgb]{0,0,0}$O_2$}%
}}}}
\end{picture}%

%% file: Cornerless.pstex_t
\begin{picture}(0,0)%
\includegraphics{Cornerless.pstex}%
\end{picture}%
\setlength{\unitlength}{789sp}%
\begingroup\makeatletter\ifx\SetFigFont\undefined%
\gdef\SetFigFont#1#2#3#4#5{%
  \reset@font\fontsize{#1}{#2pt}%
  \fontfamily{#3}\fontseries{#4}\fontshape{#5}%
  \selectfont}%
\fi\endgroup%
\begin{picture}(3666,3666)(1768,-5194)
\put(2401,-2761){\makebox(0,0)[lb]{\smash{{\SetFigFont{5}{6.0}{\rmdefault}{\mddefault}{\updefault}{\color[rgb]{0,0,0}$A$}%
}}}}
\put(4201,-2761){\makebox(0,0)[lb]{\smash{{\SetFigFont{5}{6.0}{\rmdefault}{\mddefault}{\updefault}{\color[rgb]{0,0,0}$B$}%
}}}}
\put(2401,-4561){\makebox(0,0)[lb]{\smash{{\SetFigFont{5}{6.0}{\rmdefault}{\mddefault}{\updefault}{\color[rgb]{0,0,0}$C$}%
}}}}
\put(4201,-4561){\makebox(0,0)[lb]{\smash{{\SetFigFont{5}{6.0}{\rmdefault}{\mddefault}{\updefault}{\color[rgb]{0,0,0}$D$}%
}}}}
\end{picture}%

%% file: Bigon.pstex_t
\begin{picture}(0,0)%
\includegraphics{Bigon.pstex}%
\end{picture}%
\setlength{\unitlength}{1243sp}%
\begingroup\makeatletter\ifx\SetFigFont\undefined%
\gdef\SetFigFont#1#2#3#4#5{%
  \reset@font\fontsize{#1}{#2pt}%
  \fontfamily{#3}\fontseries{#4}\fontshape{#5}%
  \selectfont}%
\fi\endgroup%
\begin{picture}(14233,8166)(-824,-5944)
\put(3961,-5551){\makebox(0,0)[lb]{\smash{{\SetFigFont{7}{8.4}{\rmdefault}{\mddefault}{\updefault}{\color[rgb]{1,0,0}3}%
}}}}
\put(-809,-2401){\makebox(0,0)[lb]{\smash{{\SetFigFont{7}{8.4}{\rmdefault}{\mddefault}{\updefault}{\color[rgb]{0,0,1}1}%
}}}}
\put(-809,-601){\makebox(0,0)[lb]{\smash{{\SetFigFont{7}{8.4}{\rmdefault}{\mddefault}{\updefault}{\color[rgb]{0,0,1}3}%
}}}}
\put(-809,-1501){\makebox(0,0)[lb]{\smash{{\SetFigFont{7}{8.4}{\rmdefault}{\mddefault}{\updefault}{\color[rgb]{0,0,1}2}%
}}}}
\put(10216,-2221){\makebox(0,0)[lb]{\smash{{\SetFigFont{7}{8.4}{\rmdefault}{\mddefault}{\updefault}{\color[rgb]{0,0,0}$O_2$}%
}}}}
\put(3061,-5551){\makebox(0,0)[lb]{\smash{{\SetFigFont{7}{8.4}{\rmdefault}{\mddefault}{\updefault}{\color[rgb]{1,0,0}2}%
}}}}
\put(1531,-2941){\makebox(0,0)[lb]{\smash{{\SetFigFont{7}{8.4}{\rmdefault}{\mddefault}{\updefault}{\color[rgb]{0,0,0}$O_1$}%
}}}}
\put(631,-3841){\makebox(0,0)[lb]{\smash{{\SetFigFont{7}{8.4}{\rmdefault}{\mddefault}{\updefault}{\color[rgb]{0,0,0}$O_2$}%
}}}}
\put(3331,-1141){\makebox(0,0)[lb]{\smash{{\SetFigFont{7}{8.4}{\rmdefault}{\mddefault}{\updefault}{\color[rgb]{0,0,0}$O_3$}%
}}}}
\put(4231,-241){\makebox(0,0)[lb]{\smash{{\SetFigFont{7}{8.4}{\rmdefault}{\mddefault}{\updefault}{\color[rgb]{0,0,0}$O_2$}%
}}}}
\put(-359,-4741){\makebox(0,0)[lb]{\smash{{\SetFigFont{7}{8.4}{\rmdefault}{\mddefault}{\updefault}{\color[rgb]{0,0,0}$O_3$}%
}}}}
\put(5131,659){\makebox(0,0)[lb]{\smash{{\SetFigFont{7}{8.4}{\rmdefault}{\mddefault}{\updefault}{\color[rgb]{0,0,0}$O_1$}%
}}}}
\put(8101,-286){\makebox(0,0)[lb]{\smash{{\SetFigFont{7}{8.4}{\rmdefault}{\mddefault}{\updefault}{\color[rgb]{0,.56,.56}$r_2$}%
}}}}
\put(5851,-1861){\makebox(0,0)[lb]{\smash{{\SetFigFont{7}{8.4}{\rmdefault}{\mddefault}{\updefault}{\color[rgb]{0,.56,.56}${\widetilde r}_2$}%
}}}}
\put(2521,-2266){\makebox(0,0)[lb]{\smash{{\SetFigFont{7}{8.4}{\rmdefault}{\mddefault}{\updefault}{\color[rgb]{0,0,0}$O_2$}%
}}}}
\put(2161,-5551){\makebox(0,0)[lb]{\smash{{\SetFigFont{7}{8.4}{\rmdefault}{\mddefault}{\updefault}{\color[rgb]{1,0,0}1}%
}}}}
\put(9181,-2851){\makebox(0,0)[lb]{\smash{{\SetFigFont{7}{8.4}{\rmdefault}{\mddefault}{\updefault}{\color[rgb]{0,0,0}$O_1$}%
}}}}
\put(10891,-1051){\makebox(0,0)[lb]{\smash{{\SetFigFont{7}{8.4}{\rmdefault}{\mddefault}{\updefault}{\color[rgb]{0,0,0}$O_3$}%
}}}}
\end{picture}%

%% file: Nonemb.pstex_t
\begin{picture}(0,0)%
\includegraphics{Nonemb.pstex}%
\end{picture}%
\setlength{\unitlength}{2486sp}%
\begingroup\makeatletter\ifx\SetFigFont\undefined%
\gdef\SetFigFont#1#2#3#4#5{%
  \reset@font\fontsize{#1}{#2pt}%
  \fontfamily{#3}\fontseries{#4}\fontshape{#5}%
  \selectfont}%
\fi\endgroup%
\begin{picture}(10804,3648)(868,-1894)
\put(9451,-781){\makebox(0,0)[lb]{\smash{{\SetFigFont{7}{8.4}{\rmdefault}{\mddefault}{\updefault}{\color[rgb]{0,0,0}$O_1$}%
}}}}
\put(9901,-331){\makebox(0,0)[lb]{\smash{{\SetFigFont{7}{8.4}{\rmdefault}{\mddefault}{\updefault}{\color[rgb]{0,0,0}$O_2$}%
}}}}
\put(10351,119){\makebox(0,0)[lb]{\smash{{\SetFigFont{7}{8.4}{\rmdefault}{\mddefault}{\updefault}{\color[rgb]{0,0,0}$O_3$}%
}}}}
\put(5041,-781){\makebox(0,0)[lb]{\smash{{\SetFigFont{7}{8.4}{\rmdefault}{\mddefault}{\updefault}{\color[rgb]{0,0,0}$O_2$}%
}}}}
\put(5941,119){\makebox(0,0)[lb]{\smash{{\SetFigFont{7}{8.4}{\rmdefault}{\mddefault}{\updefault}{\color[rgb]{0,0,0}$O_2$}%
}}}}
\put(6391,569){\makebox(0,0)[lb]{\smash{{\SetFigFont{7}{8.4}{\rmdefault}{\mddefault}{\updefault}{\color[rgb]{0,0,0}$O_3$}%
}}}}
\put(5491,-331){\makebox(0,0)[lb]{\smash{{\SetFigFont{7}{8.4}{\rmdefault}{\mddefault}{\updefault}{\color[rgb]{0,0,0}$O_1$}%
}}}}
\end{picture}%

%% file: triangles.tex
\newcommand\Lx{\widetilde\x}
\newcommand\Ly{\widetilde\y}
\newcommand\Lz{\widetilde\z}
\newcommand\Las{\widetilde\alphas}
\newcommand\Lbs{\widetilde\betas}
\newcommand\Lcs{\widetilde\gammas}
\newcommand\La{\widetilde\alpha}
\newcommand\Lb{\widetilde\beta}
\newcommand\Lc{\widetilde\gamma}
\section{Triples}
\label{sec:Triples}

We construct the Heegaard triple corresponding to wrapping.  Once
again, this will be drawn as a quotient of $\R^2$ by ${\mathbb G}$,
with the $O_i$ markings along the diagonal line with half-integer
coordinates labeled as before.  Now, we have three sets of lines,
$\Las$, $\Lbs$, and $\Lcs$. As before, $\{\La_i=i\times
\R\}_{i\in\Z}$ and  $\{\Lc_i=\R\times i\}_{i\in\Z}$ (i.e. they are the $\Lbs$ from before). We choose the $\Lbs$ with the following properties:
\begin{enumerate}[label=(${\mathcal H}$-\arabic*),ref=(${\mathcal H}$-\arabic*)]
\item The set $\Lbs$ is ${\mathbb G}_m$-invariant.
\item The slope of each line in $\Lbs$ is $-1$.
\item The ${\mathbb G}_m$-orbits of the $\Lbs$ consist of $m$ lines
  so that $\Lb_i$ and $\Lb_{i+1}$ are separated by $O_i$.
\item 
  \label{Habc:Generic}
  There is no triple-intersection point between $\La_i$, $\Lb_j$, and the diagonal.
\end{enumerate}

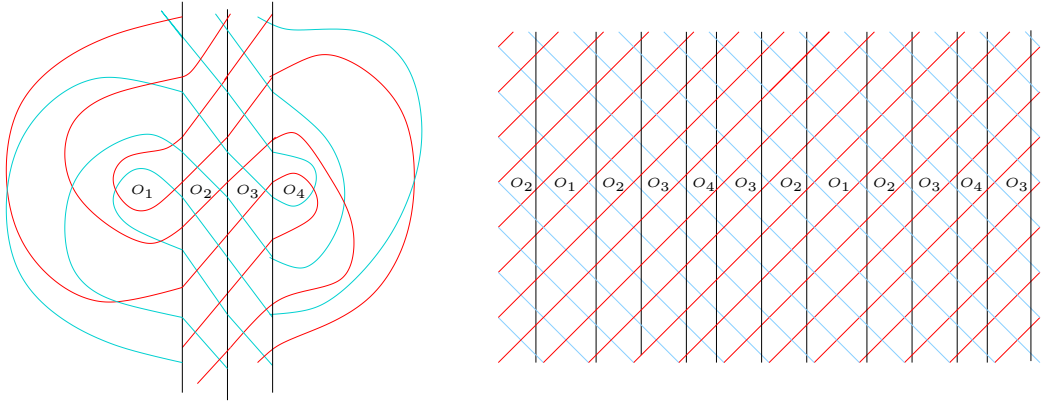
\begin{figure}[ht]
\input{HeegaardTriple.pstex_t}
\caption{\label{fig:HeegaardTriple} {\bf{Heegaard triple.}}
At the left, the (wrapped) Heegaard triple. At the right, the lift
of the diagram on the left to $\R^2$. (Note that the $O$ markings are displayed
here as horizontal; this horizontal line is to be viewed as the diagonal.)}
\end{figure}

Analogous to Definition~\ref{def:doms2}, given Heegaard states $\x$,
$\y$, and $\z$ for $\Hab$, $\Hbc$, and $\Hac$ respectively, we can
consider the set of two-chains $\psi\in \doms(\x,\y,\z)$ so that the components
of $\x$ are initial $\alpha-\beta$ corners, components of $\y$ are
initial $\beta-\gamma$ corners, and components of $\z$ are terminal
$\alpha-\gamma$ corners.

\begin{defn}
  \label{def:TriangularlyConnected}
  Let $\x$, $\y$, and $\z$ be three Heegaard states.  We say that the
  states are {\em triangularly connected} if their ${\mathbb
    G}_m$-equivariant lifts ${\widetilde \x}$, ${\widetilde \y}$, and
  ${\widetilde \z}$ in $\R^2$ admit $k$ triangles in ${\mathbb \R}^2$,
  oriented so that the they have sides in $\alpha$-$\betas$-$\gammas$
  in counterclockwise order, whose ${\mathbb G}_m$ orbits have corners
  exactly at $\Lx$, $\Ly$, and $\Lz$.  Taking the quotients of the
  triangles gives a domain $\psi\in\doms(\x,\y,\z)$.
\end{defn}

Note that there is a weaker notion: one can ask whether three Heegaard
states $\x$, $\y$, and $\z$, can be connected by a {\em Whitney
  triangle}, a continuous map from the triangle $\psi\colon T \to
\Sym^k(\HD)$, which maps three edges of the triangle to the tori
$\Ta=\alpha_1\times\dots\times \alpha_k\subset \Sym^k(\HD)$, $\Tb$,
and $\Tc$, so that the the three vertices are mapped to $\x$, $\y$,
and $\z$. For example, the elementary region containing $O_4$ in
Figure~\ref{fig:HeegaardTriple} is a triangle in $\HD$, and if we
think of its three corners as Heegaard states (with $k=1$), these
three states are connected by a Whitney triangle, which is
double-covered by the elementary hexagon containing $O_4$ on the right
in Figure~\ref{fig:HeegaardTriple}. Thus, those three states are not
triangularly connected in the sense of
Definition~\ref{def:TriangularlyConnected}, though they can be
connected by a Whitney triangle. 
Note that each Whitney triangle gives rise to a domain $\doms(\x,\y,\z)$.

\begin{lemma}
  \label{lem:TriangularConnectedMeans}
  There is a one-to-one correspondence between:
  \begin{itemize}
  \item  triples of lifted
    partial permutations $({\widetilde f},{\widetilde S})$,
    $({\widetilde g},{\widetilde T})$ with ${\widetilde T}={\widetilde
      f}({\widetilde S})$ and $({\widetilde g}\circ {\widetilde
      f},{\widetilde S})$
  \item
    triangularly connected triples of 
    Heegaard states $\x\in\States(\Hab)$, $\y\in\States(\Hbc)$,
    $\z\in\States(\Hac)$.
  \end{itemize}
    Moreover, 
    under this correspondence
    \begin{equation}
      \label{eq:WeightEquation}
    \weight_i({\widetilde f},{\widetilde S}) 
    + \weight_i({\widetilde g},{\widetilde T})
    = 
    \weight_i({\widetilde g}\circ{\widetilde f},{\widetilde S})+
    \# (O_i\cap \psi).
    \end{equation}
\end{lemma}

\begin{proof}
  Given $\x$, $\y$, and $\z$,
  their corresponding lifted partial permutations ${\widetilde f}$,
  ${\widetilde g}$, and ${\widetilde h}$ are characterized by
  \[ 
    {\widetilde \x}=\bigcup \La_{s}\cap \Lb_{{\widetilde f}(s)},
    \qquad {\widetilde \y}=\bigcup \Lb_{t}\cap \Lb_{{\widetilde
        g}(s)}, \qquad {\widetilde \x}=\bigcup \La_{s}\cap
    \Lb_{{\widetilde h}(s)}.\] Thus, the condition that the components
    of $\Lx$, $\Ly$, and $\Lz$ can be connected by triangles is
    precisely the condition that ${\widetilde h}(s)={\widetilde g}\circ
    {\widetilde f}$.

  To establish Equation~\eqref{eq:WeightEquation}, we argue as follows.

  Consider a coordinate $x\in \La_i\cap\Lb_j$. By
  Condition~\ref{Habc:Generic}, the lines $\La_i$ and $\Lb_j$, and the
  diagonal divide $\R^2$ into seven regions, one of which is a
  compact region -- indeed, it is a triangle triangle $T_x$, Let
  $O_i(T_x)$ be $\OneHalf$ times the number of $O_i$ appears in
  $T_x$. (Note that each occurence of $O_i$ appears on the boundary of
  $T_i$, hence the factor of $\OneHalf$.)
  
  Given a Heegaard state ${\widetilde \x}$, choose any 
  unordered set of $m$ points
  $\{x_1,\dots,x_m\}\subset \R^2$ whose ${\mathbb G}_m$-orbit is $\x$.
  It is elementary to see that 
  \[ \weight_i(\x)=\sum_{j=1}^m  O_i(T_{x_j}).\]

  Equation~\eqref{eq:WeightEquation} is obtained by counting points in
  the plane, divided into cases according to how the diagonal line
  intersects each triangle. Specifically, after applying an element of
  ${\mathbb G}_m$ if necessary, we can assume that the
  $\alpha$-$\gamma$ corner of the triangle is on the upper right.
  There are now four remaining cases, according to the
  number of vertices of the triangle which lie above the diagonal
  line: this can be any number between $0$ and $3$. 
  The cases are illustrated in Figure~\ref{fig:CountOs}.

  For the case on the left on that figure (where the triangle is
  entirely below the diagonal line), the weight of $\x$ counts the $O$
  markings on the diagonal boundary of $B\cup C$ or, equivalently,
  $C$; the weight of $\y$ counts the markings on the diagonal boundary
  of $A$; and the weigth of $\z$ counts $O$ markings on the diagonal
  boundary of $A\cup B$; more succinctly, 
  \[ 
    \weight_i(\x)= O_i (B)+ O_i(C) \qquad
    \weight_i(\y)= O_i(A) \qquad
    \weight_i(\z)= O_i(A)+ O_i(B).
    \]
    Since in this case the number of $O_i$
    in the triangle is given by $ O_i(C)=0$,
    Equation~\eqref{eq:WeightEquation} follows.  The other three cases
    are:
  \[ 
    \weight_i(\x)= O_i (B)+ O_i(C) \qquad
    \weight_i(\y)= O_i(A) \qquad
    \weight_i(\z)= O_i(B).
    \]
  \[ 
    \weight_i(\x)= O_i (C) \qquad
    \weight_i(\y)= O_i(A)+ O_i(B) \qquad
    \weight_i(\z)= O_i(B).
    \]
  \[ 
    \weight_i(\x)= O_i (B) \qquad
    \weight_i(\y)= O_i(A)+ O_i(C) \qquad
    \weight_i(\z)= O_i(B)+ O_i(C);
    \]
    and the number of $O_i$ markings in the triangles are
    \[  O_i(A)+ O_i(C)\qquad  O_i(A)+ O_i(C)\qquad  O_i(A)=0 \]
    respectively. Thus, in the remaining three cases, Equation~\eqref{eq:WeightEquation} holds.
\begin{figure}[ht]
\input{CountOs.pstex_t}
\caption{\label{fig:CountOs} {\bf{Counting $O$'s in triangles.}}
The diagonal dashed line represents the diagonal.}
\end{figure}
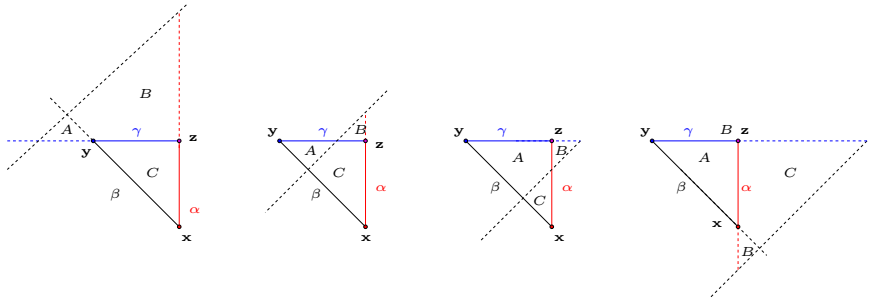
\end{proof}

Consider two-chains for $\HD$ analogous to Definition~\ref{def:doms2},
now using all three sets of curves $\alphas$, $\betas$, and $\gammas$.
These curves divide $\HD$ into {\em elementary domains}, which are
polygons. As in~\cite{HolDisk,RasmussenThesis,LipshitzCyl}, the {\em
  euler measure} each polygon has an {\em Euler measure}: if the
number of sides is $m$, the euler measure of the corresponding polygon
is given by $1-\frac{m}{4}$. Extend the Euler measure $e$ linearly to
all two-chains, and denote the resulting function $e$.

\begin{figure}[ht]
\input{Triangles.pstex_t}
\caption{\label{fig:Triangles} {\bf{Triangles.}}
At the left: a triangle with Euler measure $3/4$; at the right, a triangle with Euler measure $1/4$.}
\end{figure}
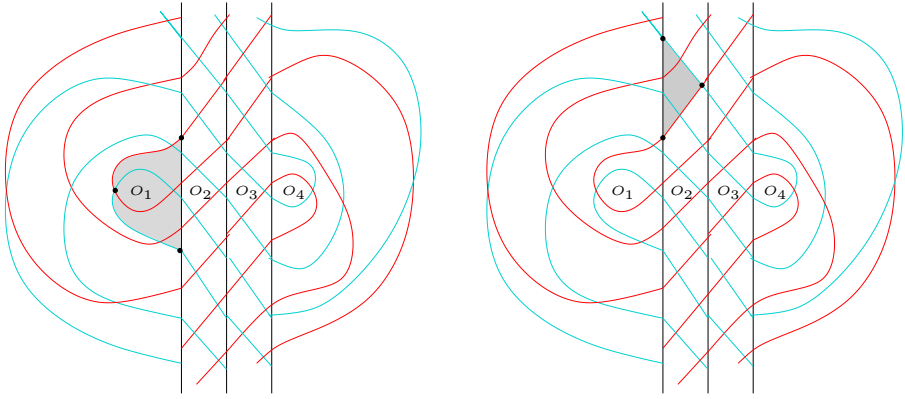

\begin{lemma}
  \label{lem:EulerMeasure}
  Let $\psi\in\pi_2(\x,\y,\z)$ be a positive domain. Then, the euler
  measure of $\psi$ is computed by
  \begin{equation}
    \label{eq:EulerMeasure}
    e(\psi)= \frac{k}{4}+\frac{ O_1(\psi)}{2} + \frac{ O_m(\psi)}{2}.
  \end{equation}
\end{lemma}

\begin{proof}
  Loosely speaking, we claim that any positive domain can be cut along
  the $\alpha$, $\beta$, and $\gamma$ lines to give $k$ triangles
  (each with Euler measure $1/4$), $O_i(\psi)$+$O_m(\psi)$ bigons
  (each with Euler measure $1/2$), and many rectangles (each with
  Euler measure $0$). We make this claim precise as follows.

  Consider the group of cornerless $2$-chains with compact support,
  generalizing Definition~\ref{def:doms2} in a straightforward way.
  This group is evidently $0$. But there is a non-trivial group ${\mathcal A}$ of
  $2$-chains which are allowed $\beta$-$\gamma$ and $\alpha$-$\gamma$ 
  corners, but no $\alpha$-$\beta$ corners. More precisely, at each intersection
  of $\alpha$ with $\beta$, we require that the alternating sum of
  the local multiplicities of the four quadrants add up to zero; i.e.
  $A+D=B+C$, in the conventions of 
  Figure~\ref{fig:Cornerless}.

  That group is non-trivial: for example, fix a $\beta$- or $\gamma$-segment
  that
  connects a pair of consecutive $\beta$-lines; and fix another
  $\beta$- or $\gamma$-segment
  that connects the same pair of $\beta$-lines. The four
  segments enclose some region in the plane: when all four segments are disjoint
  that region is a quadrilateral, when
  they intersect, it is a difference of two triangles. See
  Figure~\ref{fig:Fundamentals}.  
  When it
  happens that these four segments are permuted by some ${\mathbb
    G}_m$ action, We call these regions {\em fundamental bigons};
  otherwise, we call them {\em fundamental rectangles}. 
  Fundamental bigons project to actual bigons in $\HD$;
  while fundamental rectangles project to rectangles (or differences of two 
  triangles).

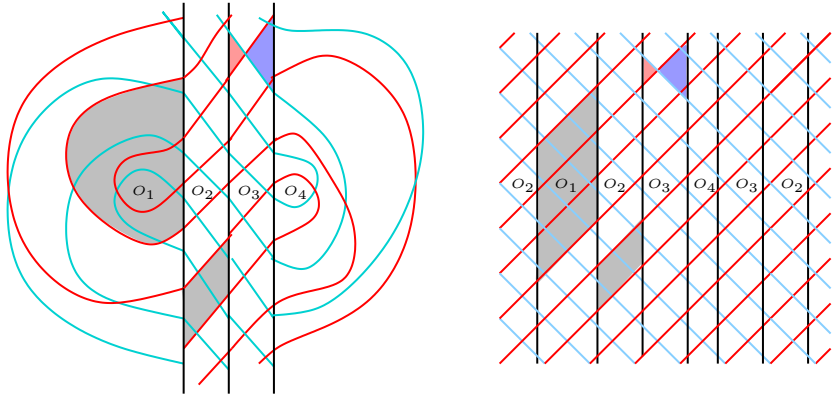
\begin{figure}[ht]
\input{Fundamentals4.pstex_t}
\caption{\label{fig:Fundamentals} {\bf{Fundamental regions.}}  At the
  left, we have
  a fundamental bigon, a fundamental rectangle, and a difference of
  two triangles. At the right, we have their lifts.)}
\end{figure}

It is elementary to verify that ${\mathcal A}$ is generated by the
fundamental bigons and rectangles. Moreover, the quantity $e(D)-
\frac{O_1(D)}{2}-\frac{O_m(D)}{2}$ vanishes on all fundamental bigons
and rectangles.

Let ${\mathcal B}$ the space of domains with
$\alpha$-$\gamma$-terminal corner (in the sense of
Definition~\ref{def:doms2}) at some component of  $\Lz$. This is clearly an affine
space of ${\mathcal A}$. Moreover, given any $\Lz$, we can draw some
positive union of triangles $T_0(\z)\in B$. Clearly, the Euler measure
is $k/4$.  Now, the map ${\widetilde e}=e-\frac{O_1}{2}-\frac{O_m}{2}\colon \doms(\x,\y,\z)\to \Z/4$ factors
through the space of ${\mathcal A}$-orbits in ${\mathcal B}$,
${\mathcal B}/{\mathcal A}$, since ${\widetilde e}$ vanishes on 
${\mathcal A}$. Since $\psi$ and $T_0(\z)$ have the the same
image in ${\mathcal B}/{\mathcal A}$, and $e(T_0(\x))=k/4$, it follows
that $e(\psi)=k/4$.
\end{proof}

The relevance of Lemma~\ref{lem:EulerMeasure} stems from Sarkar's
computation of the Maslov index of a triangle (or more generally, a
Whitney $n$-gon)~\cite{SarkarMaslov}.  One of his formulas, generalizing a
theorem of Rasmussen when $n=2$~\cite{RasmussenThesis}, gives

\begin{equation}
  \label{eq:IndexOfNgon}
  \Mas(\psi)=2e(\psi)-\frac{k(n-2)}{2}+\#(\psi\cap \Delta).
\end{equation}

Here, $\#(\psi\cap\Delta)$ denotes the algebraic intersection number of of
$\psi$ with the big diagonal in $\Sym^k(\HD)$.

\begin{lemma}
  \label{lem:MasZeroTriangConn}
  If $\psi\in\doms(\x,\y,\z)$ has $\Mas(\psi)=0$, and $\psi$ has a
  pseudo-holomorphic representative, then $\x$, $\y$, and $\z$ are
  triangularly connected, in the sense of
  Definition~\ref{def:TriangularlyConnected}.
\end{lemma}

\begin{proof}
We begin with some remarks. Let $\x=\{x_1,\dots,x_k\}$,
$\y=\{y_1,\dots,y_k\}$, $\z=\{z_1,\dots,z_k\}$.  Suppose there are $k$
triangles $\psi_i\colon T\to \HD$ with corners at $x_i$, $y_i$, and
$z_i$, and edges mapping to some $\alpha$-curve, some $\beta$-curve,
and some $\gamma$-curve.  We can then form
$\psi=\psi_1\times\dots\times \psi_k\colon T \to \Sym^k(\HD)$, to get
a Whitney triangle. We call such
Whitney triangles {\em decomposable}. Not every Whitney triangle is
decomposable, but Whitney triangles which are disjoint from the
diagonal $\Delta$ are.

  Now, combining Equation~\eqref{eq:EulerMeasure} with
  Equation~\eqref{eq:IndexOfNgon}, we see that for any triangle,
  \[ \Mas(\psi)=O_1(\psi) + O_m(\psi) + \#(\psi\cap \Delta).\]
  If $\psi$ has a pseudo-holomorphic representative, all three of the
  terms on the right are non-negative.  Indeed, if $\Mas(\psi)=0$,
  then $\#(\psi\cap \Delta)=0$; and indeed, $\psi$ is disjoint from
  the diagonal; and hence it is decomposable.  Next, observe that
  since $O_1(\psi)=O_m(\psi)=0$, each factor $\psi_i\colon T\to \HD$
  in the factor of decomposition of $\psi$ maps to
  $\HD\setminus\{O_1,O_m\}$; i.e. the locus where the quotient map
  $\R^2\to \HD=\R^2/{\mathbb G}_m$ is a covering space.  Thus, we can
  lift each factor $\psi_i$ to obtain maps $\{{\widetilde \psi}_i\colon
  T \to \R^2\}_{i=1}^k$, showing that $\x$, $\y$ and $\z$ are triangularly connected.
\end{proof}

\begin{lemma}
  \label{lem:DiagonalCount}
  Suppose that $\x$, $\y$, and $\z$ are triangularly connected Heegaard
  states, let $\psi\in\doms(\x,\y,\z)$ be the domain in $\HD$ 
  connecting them, and
  choose triangles $\{\widetilde \psi_i\colon T \to \R^2\}_{i=1}^k$ 
  whose projection to $\HD$ gives $\psi\in\doms(\x,\y,\z)$.
  Then, 
  \begin{equation}
    \label{eq:DiagonalCount}
    \#(\psi\cap \Delta) =  
    \OneHalf \sum_{i,j}^k \sum_{g\in {\mathbb G}_m} \#\left({\widetilde \psi}_i
    \times g\cdot {\widetilde \psi}_j\right)\cap \Delta,
  \end{equation}
  where the right-hand-side is to be interpreted as an intersection 
  with $\Delta\subset \Sym^2(\R^2)$.
\end{lemma}

\begin{proof}
  The intersection of $\psi=\bigtimes_{i=1}^k \psi_i$ with $\Delta$ in
  $\Sym^k(\HD)$ is identified with ${\mathbb G}_m$-orbits of
  data $g_1,
  g_2\in{\mathbb G}_m$, $\tau_1,\tau_2\in {\mathbb R}^2$ with
  $g_1\cdot {\widetilde \psi}_i(\tau_1)=g_2\cdot {\widetilde
    \psi}_j(\tau_2)$, where the pair of triples
  $(g_1,\tau_1,{\widetilde \psi}_i)$ and $(g_2,\tau_2,{\widetilde
    \psi}_j)$ is thought of as unordered.  Breaking symmetry, we can
  think of this as half the count of $\tau_1,\tau_2\in \R^2$ and $g\in
  {\mathbb G}_m$ and pairs $i$, $j$, so that
  ${\widetilde\psi}_i(\tau_1)=g\cdot {\widetilde\psi}_j(\tau_2)$. This
  is the count on the right-hand-side of
  Equation~\eqref{eq:DiagonalCount}.
\end{proof}

\begin{lemma}
  \label{lem:MasZeroTriangleMult}
  Suppose that $\x$, $\y$, and $\z$ are three triangularly connected
  Heegaard states, corresponding to lifted partial permutations Let
  $({\widetilde f},{\widetilde S})$, $({\widetilde f},{\widetilde
    T})$, and $({\widetilde g}\circ {\widetilde f},{\widetilde S})$.
  Let $\psi\in \doms(\x,\y,\z)$ be the corresponding domain
  in $\Habc$.  Suppose moreover that
  $O_1(\psi)=O_m(\psi)=0$.  Then,
    \begin{equation}
      \label{eq:TriangleMaslov}
      \cross({\widetilde f},{\widetilde S})+
    \cross({\widetilde g},{\widetilde T})
    -\cross({\widetilde g}\circ {\widetilde f},{\widetilde S}) 
    = \Mas(\psi).
    \end{equation}
\end{lemma}

\begin{proof}
  Each strand in ${\widetilde f}$ corresponds to a Heegaard state for $\Hab$;
  each strand in ${\widetilde g}$ corresponds to a Heegaard state for $\Hbc$,
  and each composite strand corresponds to a triangle in $\Habc$,
  which in turn corresponds to a ${\mathbb G}_m$-orbit of a triangle in $\R^2$.
  Up to the action of ${\mathbb G}_m$, we can assume that the triangle has its
  $\alpha-\gamma$ vertex above its $\beta$-line, as in Figure~\ref{fig:MaslovComputation}.
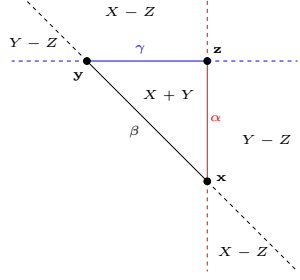
\begin{figure}[ht]
\input{MaslovComputation.pstex_t}
\caption{\label{fig:MaslovComputation} {\bf{Maslov index of triangles.}}}
\end{figure}

Extending the lines of each triangle, we see that each triangle splits
the plane into $7$ regions, one of which is the (compact) triangle
itself.  The left-hand-side of Equation~\eqref{eq:TriangleMaslov} can
be interpreted as a sum, over each triangle $T$ of a count of all the
other components of type $X$, $Y$, or $Z$, in the seven regions,
counted with multiplicity $\pm 1$ or $0$, as indicated in
Figure~\ref{fig:MaslovComputation}.  This count, in turn, can be
organized according to all other triangles $T'$ connecting three
auxiliary generators, taken with given multiplicity.

Note that $T'$ is the image of $T$ under an an affine transformation
$L$ of $\R^2$ which is a composition of a real rescaling (which
includes a possible $180^\circ$ rotation) composed with a
translation. Note that $L$ is either a translation or it has a unique
fixed point. 

We claim that $T\cup T'$ hits the diagonal precisely when $L$ has a
fixed point, and that fixed point is contained in the interior of $L$
(or $L'$); and . moreover, in that case $\#((\psi_T\times \psi_T')\cap
\Delta)=2$. To see this, note that the map to $\Sym^2(\R^2)$ is
modeled on the map $t\mapsto \{t,L(t)\}$, which in turn corresponds to
the monic polynomial $z^2 -(t+L(t)) z +t L(t)$, whose discriminant is
$(t+L(t))^2-4 tL(t)=(t-L(t))^2$, which vanishes to order $2$ at the fixed
point of $L$.

We now verify that the contribution of the triangle pair $T,T'$ to the
left-hand-side of Equation~\eqref{eq:TriangleMaslov} coincides with
this intersection number with the diagonal, by looking at the possible
cases for the two triangles. After possibly switching roles of $T$ and
$T'$, we can assume that the number of vertices of $T'$ in $T$ is
greater or equal to the number of vertices of $T$ in $T'$. There are
the following possibilities:
\begin{itemize}
\item $T$ contains no vertices of $T'$.
  In this case, either $T'$ and $T$ are disjoint, or 
  they can overlap as pictured in 
  the first picture of Figure~\ref{fig:MaslovComputation5},
  in which case the local contribution is $2$.
\item $T$ contains exactly one vertex of $T'$.
  This can happen in two inequivalent ways: either $T'$ is a translate of $T$,
  in which case the contribution is $0$;
  or $T'$ is not obtained as a translate of $T$, in which case
  the local contribution is $2$, as can be seen by considering the six cases
  in the second picture of Figure~\ref{fig:MaslovComputation5}
\item $T$ contains exactly two vertices of $T'$ in its interior, as
  shown in the three cases on the third picture of
  Figure~\ref{fig:MaslovComputation5}.
\item $T$ contains all three vertices of $T'$ in its interior,
  as shown in the fourth picture of Figure~\ref{fig:MaslovComputation},
  in which case the local contribution is $2$.
\end{itemize}

\begin{figure}[ht]
\input{MaslovComputation5.pstex_t}
\caption{\label{fig:MaslovComputation5} {\bf{Triangle pairs.}}  Here
  are the various combinatorial ways two triangles can interact, so
  that their contributions to (both sides of)
  Equation~\eqref{eq:TriangleMaslov} are non-zero. (Think of $T$ as
  the large triangle; various choices of the other triangle $T'$ are
  indicated by the smaller triangles.)}
\end{figure}
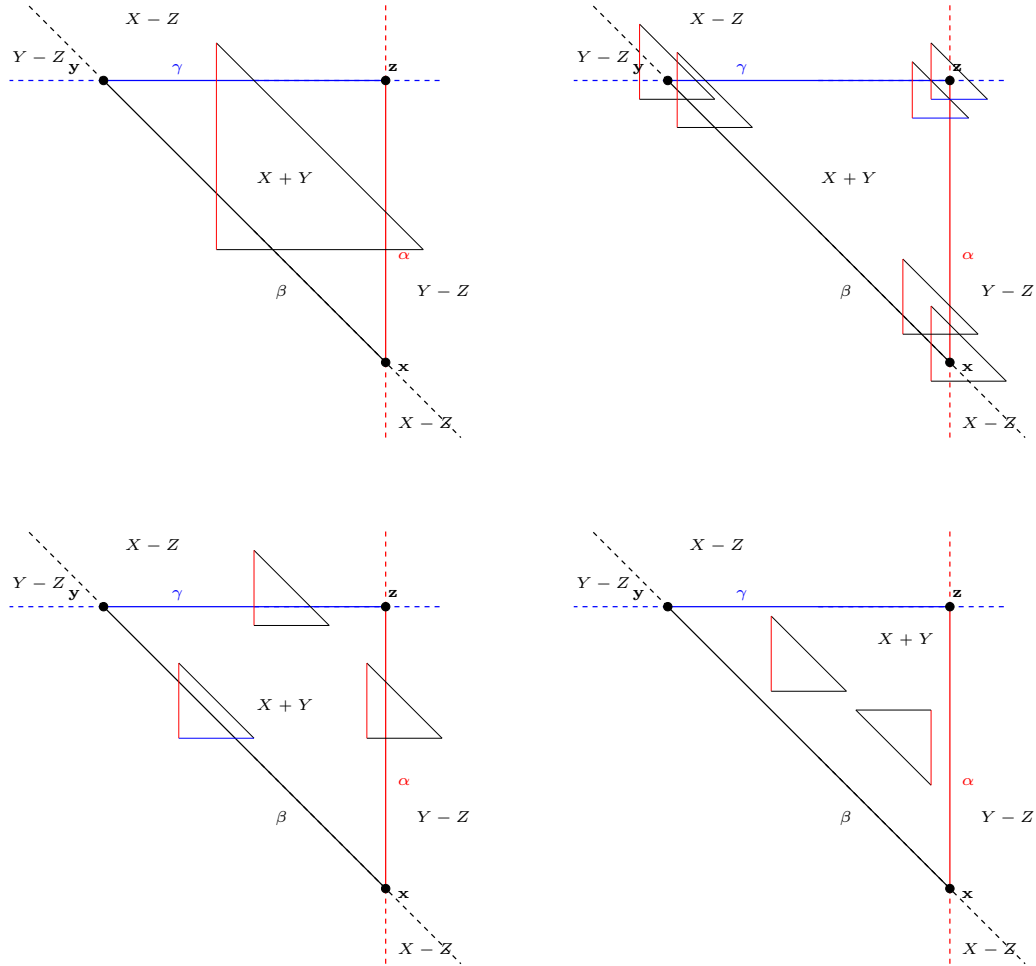

In view of  Lemma~\ref{lem:DiagonalCount}, we have verified that
\[ 
      \cross({\widetilde f},{\widetilde S})+ \cross({\widetilde
        g},{\widetilde T}) -\cross({\widetilde g}\circ {\widetilde
        f},{\widetilde S}) = \#(\Delta\cap \psi).\]

      Combining Equations~\eqref{eq:EulerMeasure} and~\ref{eq:IndexOfNgon},
      we see that 
      \[ \Mas(\psi)=O_1(\psi)+O_m(\psi)+\#(\psi\cap \Delta);\]
      since $O_1(\psi)=O_m(\psi)=0$, Equation~\eqref{eq:TriangleMaslov}.
\end{proof}

\begin{prop}
  \label{prop:TrianglePong}
  Under the identification from Proposition~\ref{prop:IdentifyComplexes},
  the triangle map corresponds to composition in the pong algebra.
\end{prop}

\begin{proof}
  The triangle map counts index zero triangles, which corresponds to
  triangularly connected lifted partial permutations by
  Lemma~\ref{lem:MasZeroTriangConn}.
  Lemma~\ref{lem:MasZeroTriangleMult} in turn identifies the counts of
  index zero triangles with compositions in the pong algebra,
  in view of Lemma~\ref{lem:NoOuterUs}. Moreover, 
  Equation~\eqref{eq:WeightEquation} identifies the coefficients of
  those counts.

  Implicit in the above identification is the identification
  $(\HD,\alphas,\betas,\{O_1,\dots,O_m\})\cong
  (\HD,\alphas,\betas,\{O_1,\dots,O_m\})$, which can be viewed as the
  identification coming from the Liouville flow. (Compare
  Equation~\eqref{eq:ImageUnderLiouville}.)  The corresponding map on
  the wrapped complexes was denoted $\sigma$ in
  Equation~\eqref{lem:TriangleMap}. To promote this to a chain map, in
  general, following~\cite{AbouzaidCriterion}, one must use a
  continuation map interpolating between the complex structure used on
  $\HD$, and its pull-back under the Liouville flow.

  For the Heegaard diagram $(\HD,\alphas,\betas,\{O_i\}_{i=1}^m)$,
  though the continuation map between any two admissible paths of
  almost-complex structures is simply the identity map, since the
  continuation is supported on non-negative domains with index zero
  and, following Lemma~\ref{lem:PositiveDomains}
  (i.e. Property~\ref{P:ResolveCrossings} combined with
  Equation~\eqref{eq:MaslovDifference}), the only such domain is the
  constant domain (with $\x=\y$).
\end{proof}

%% file: HeegaardTriple.pstex_t
\begin{picture}(0,0)%
\includegraphics{HeegaardTriple.pstex}%
\end{picture}%
\setlength{\unitlength}{829sp}%
\begingroup\makeatletter\ifx\SetFigFont\undefined%
\gdef\SetFigFont#1#2#3#4#5{%
  \reset@font\fontsize{#1}{#2pt}%
  \fontfamily{#3}\fontseries{#4}\fontshape{#5}%
  \selectfont}%
\fi\endgroup%
\begin{picture}(30981,11991)(-5747,-10669)
\put(9271,-4246){\makebox(0,0)[lb]{\smash{{\SetFigFont{5}{6.0}{\rmdefault}{\mddefault}{\updefault}{\color[rgb]{0,0,0}$O_2$}%
}}}}
\put(-314,-4471){\makebox(0,0)[lb]{\smash{{\SetFigFont{5}{6.0}{\rmdefault}{\mddefault}{\updefault}{\color[rgb]{0,0,0}$O_2$}%
}}}}
\put(1081,-4471){\makebox(0,0)[lb]{\smash{{\SetFigFont{5}{6.0}{\rmdefault}{\mddefault}{\updefault}{\color[rgb]{0,0,0}$O_3$}%
}}}}
\put(2476,-4471){\makebox(0,0)[lb]{\smash{{\SetFigFont{5}{6.0}{\rmdefault}{\mddefault}{\updefault}{\color[rgb]{0,0,0}$O_4$}%
}}}}
\put(-2069,-4471){\makebox(0,0)[lb]{\smash{{\SetFigFont{5}{6.0}{\rmdefault}{\mddefault}{\updefault}{\color[rgb]{0,0,0}$O_1$}%
}}}}
\put(14716,-4246){\makebox(0,0)[lb]{\smash{{\SetFigFont{5}{6.0}{\rmdefault}{\mddefault}{\updefault}{\color[rgb]{0,0,0}$O_4$}%
}}}}
\put(10576,-4246){\makebox(0,0)[lb]{\smash{{\SetFigFont{5}{6.0}{\rmdefault}{\mddefault}{\updefault}{\color[rgb]{0,0,0}$O_1$}%
}}}}
\put(12016,-4246){\makebox(0,0)[lb]{\smash{{\SetFigFont{5}{6.0}{\rmdefault}{\mddefault}{\updefault}{\color[rgb]{0,0,0}$O_2$}%
}}}}
\put(13366,-4246){\makebox(0,0)[lb]{\smash{{\SetFigFont{5}{6.0}{\rmdefault}{\mddefault}{\updefault}{\color[rgb]{0,0,0}$O_3$}%
}}}}
\put(15976,-4246){\makebox(0,0)[lb]{\smash{{\SetFigFont{5}{6.0}{\rmdefault}{\mddefault}{\updefault}{\color[rgb]{0,0,0}$O_3$}%
}}}}
\put(21466,-4246){\makebox(0,0)[lb]{\smash{{\SetFigFont{5}{6.0}{\rmdefault}{\mddefault}{\updefault}{\color[rgb]{0,0,0}$O_3$}%
}}}}
\put(17326,-4246){\makebox(0,0)[lb]{\smash{{\SetFigFont{5}{6.0}{\rmdefault}{\mddefault}{\updefault}{\color[rgb]{0,0,0}$O_2$}%
}}}}
\put(18766,-4246){\makebox(0,0)[lb]{\smash{{\SetFigFont{5}{6.0}{\rmdefault}{\mddefault}{\updefault}{\color[rgb]{0,0,0}$O_1$}%
}}}}
\put(20116,-4246){\makebox(0,0)[lb]{\smash{{\SetFigFont{5}{6.0}{\rmdefault}{\mddefault}{\updefault}{\color[rgb]{0,0,0}$O_2$}%
}}}}
\put(22726,-4246){\makebox(0,0)[lb]{\smash{{\SetFigFont{5}{6.0}{\rmdefault}{\mddefault}{\updefault}{\color[rgb]{0,0,0}$O_4$}%
}}}}
\put(24076,-4246){\makebox(0,0)[lb]{\smash{{\SetFigFont{5}{6.0}{\rmdefault}{\mddefault}{\updefault}{\color[rgb]{0,0,0}$O_3$}%
}}}}
\end{picture}%

%% file: CountOs.pstex_t
\begin{picture}(0,0)%
\includegraphics{CountOs.pstex}%
\end{picture}%
\setlength{\unitlength}{592sp}%
\begingroup\makeatletter\ifx\SetFigFont\undefined%
\gdef\SetFigFont#1#2#3#4#5{%
  \reset@font\fontsize{#1}{#2pt}%
  \fontfamily{#3}\fontseries{#4}\fontshape{#5}%
  \selectfont}%
\fi\endgroup%
\begin{picture}(36066,12366)(-9332,-8494)
\put(18601,-3961){\makebox(0,0)[lb]{\smash{{\SetFigFont{5}{6.0}{\rmdefault}{\mddefault}{\updefault}{\color[rgb]{0,0,0}$\beta$}%
}}}}
\put(19501,-2761){\makebox(0,0)[lb]{\smash{{\SetFigFont{5}{6.0}{\rmdefault}{\mddefault}{\updefault}{\color[rgb]{0,0,0}$A$}%
}}}}
\put(23101,-3361){\makebox(0,0)[lb]{\smash{{\SetFigFont{5}{6.0}{\rmdefault}{\mddefault}{\updefault}{\color[rgb]{0,0,0}$C$}%
}}}}
\put(21301,-6661){\makebox(0,0)[lb]{\smash{{\SetFigFont{5}{6.0}{\rmdefault}{\mddefault}{\updefault}{\color[rgb]{0,0,0}$B$}%
}}}}
\put(3601,-1561){\makebox(0,0)[lb]{\smash{{\SetFigFont{5}{6.0}{\rmdefault}{\mddefault}{\updefault}{\color[rgb]{0,0,1}$\gamma$}%
}}}}
\put(6001,-3961){\makebox(0,0)[lb]{\smash{{\SetFigFont{5}{6.0}{\rmdefault}{\mddefault}{\updefault}{\color[rgb]{1,0,0}$\alpha$}%
}}}}
\put(3301,-4261){\makebox(0,0)[lb]{\smash{{\SetFigFont{5}{6.0}{\rmdefault}{\mddefault}{\updefault}{\color[rgb]{0,0,0}$\beta$}%
}}}}
\put(10801,-3961){\makebox(0,0)[lb]{\smash{{\SetFigFont{5}{6.0}{\rmdefault}{\mddefault}{\updefault}{\color[rgb]{0,0,0}$\beta$}%
}}}}
\put(11101,-1561){\makebox(0,0)[lb]{\smash{{\SetFigFont{5}{6.0}{\rmdefault}{\mddefault}{\updefault}{\color[rgb]{0,0,1}$\gamma$}%
}}}}
\put(13801,-3961){\makebox(0,0)[lb]{\smash{{\SetFigFont{5}{6.0}{\rmdefault}{\mddefault}{\updefault}{\color[rgb]{1,0,0}$\alpha$}%
}}}}
\put(18901,-1561){\makebox(0,0)[lb]{\smash{{\SetFigFont{5}{6.0}{\rmdefault}{\mddefault}{\updefault}{\color[rgb]{0,0,1}$\gamma$}%
}}}}
\put(21301,-3961){\makebox(0,0)[lb]{\smash{{\SetFigFont{5}{6.0}{\rmdefault}{\mddefault}{\updefault}{\color[rgb]{1,0,0}$\alpha$}%
}}}}
\put(-4199,-1561){\makebox(0,0)[lb]{\smash{{\SetFigFont{5}{6.0}{\rmdefault}{\mddefault}{\updefault}{\color[rgb]{0,0,1}$\gamma$}%
}}}}
\put(-1799,-4861){\makebox(0,0)[lb]{\smash{{\SetFigFont{5}{6.0}{\rmdefault}{\mddefault}{\updefault}{\color[rgb]{1,0,0}$\alpha$}%
}}}}
\put(9301,-1561){\makebox(0,0)[lb]{\smash{{\SetFigFont{5}{6.0}{\rmdefault}{\mddefault}{\updefault}{\color[rgb]{0,0,0}$\y$}%
}}}}
\put(13501,-1561){\makebox(0,0)[lb]{\smash{{\SetFigFont{5}{6.0}{\rmdefault}{\mddefault}{\updefault}{\color[rgb]{0,0,0}$\z$}%
}}}}
\put(17101,-1561){\makebox(0,0)[lb]{\smash{{\SetFigFont{5}{6.0}{\rmdefault}{\mddefault}{\updefault}{\color[rgb]{0,0,0}$\y$}%
}}}}
\put(21301,-1561){\makebox(0,0)[lb]{\smash{{\SetFigFont{5}{6.0}{\rmdefault}{\mddefault}{\updefault}{\color[rgb]{0,0,0}$\z$}%
}}}}
\put(20401,-1561){\makebox(0,0)[lb]{\smash{{\SetFigFont{5}{6.0}{\rmdefault}{\mddefault}{\updefault}{\color[rgb]{0,0,0}$B$}%
}}}}
\put(1501,-1561){\makebox(0,0)[lb]{\smash{{\SetFigFont{5}{6.0}{\rmdefault}{\mddefault}{\updefault}{\color[rgb]{0,0,0}$\y$}%
}}}}
\put(5401,-6061){\makebox(0,0)[lb]{\smash{{\SetFigFont{5}{6.0}{\rmdefault}{\mddefault}{\updefault}{\color[rgb]{0,0,0}$\x$}%
}}}}
\put(-2099,-6061){\makebox(0,0)[lb]{\smash{{\SetFigFont{5}{6.0}{\rmdefault}{\mddefault}{\updefault}{\color[rgb]{0,0,0}$\x$}%
}}}}
\put(-6299,-2461){\makebox(0,0)[lb]{\smash{{\SetFigFont{5}{6.0}{\rmdefault}{\mddefault}{\updefault}{\color[rgb]{0,0,0}$\y$}%
}}}}
\put(-1799,-1861){\makebox(0,0)[lb]{\smash{{\SetFigFont{5}{6.0}{\rmdefault}{\mddefault}{\updefault}{\color[rgb]{0,0,0}$\z$}%
}}}}
\put(-3599,-3361){\makebox(0,0)[lb]{\smash{{\SetFigFont{5}{6.0}{\rmdefault}{\mddefault}{\updefault}{\color[rgb]{0,0,0}$C$}%
}}}}
\put(-3899,-61){\makebox(0,0)[lb]{\smash{{\SetFigFont{5}{6.0}{\rmdefault}{\mddefault}{\updefault}{\color[rgb]{0,0,0}$B$}%
}}}}
\put(-7199,-1561){\makebox(0,0)[lb]{\smash{{\SetFigFont{5}{6.0}{\rmdefault}{\mddefault}{\updefault}{\color[rgb]{0,0,0}$A$}%
}}}}
\put(5101,-1561){\makebox(0,0)[lb]{\smash{{\SetFigFont{5}{6.0}{\rmdefault}{\mddefault}{\updefault}{\color[rgb]{0,0,0}$B$}%
}}}}
\put(4201,-3361){\makebox(0,0)[lb]{\smash{{\SetFigFont{5}{6.0}{\rmdefault}{\mddefault}{\updefault}{\color[rgb]{0,0,0}$C$}%
}}}}
\put(-5099,-4261){\makebox(0,0)[lb]{\smash{{\SetFigFont{5}{6.0}{\rmdefault}{\mddefault}{\updefault}{\color[rgb]{0,0,0}$\beta$}%
}}}}
\put(13501,-6061){\makebox(0,0)[lb]{\smash{{\SetFigFont{5}{6.0}{\rmdefault}{\mddefault}{\updefault}{\color[rgb]{0,0,0}$\x$}%
}}}}
\put(20101,-5461){\makebox(0,0)[lb]{\smash{{\SetFigFont{5}{6.0}{\rmdefault}{\mddefault}{\updefault}{\color[rgb]{0,0,0}$\x$}%
}}}}
\put(6001,-2161){\makebox(0,0)[lb]{\smash{{\SetFigFont{5}{6.0}{\rmdefault}{\mddefault}{\updefault}{\color[rgb]{0,0,0}$\z$}%
}}}}
\put(13501,-2461){\makebox(0,0)[lb]{\smash{{\SetFigFont{5}{6.0}{\rmdefault}{\mddefault}{\updefault}{\color[rgb]{0,0,0}$B$}%
}}}}
\put(11701,-2761){\makebox(0,0)[lb]{\smash{{\SetFigFont{5}{6.0}{\rmdefault}{\mddefault}{\updefault}{\color[rgb]{0,0,0}$A$}%
}}}}
\put(3001,-2461){\makebox(0,0)[lb]{\smash{{\SetFigFont{5}{6.0}{\rmdefault}{\mddefault}{\updefault}{\color[rgb]{0,0,0}$A$}%
}}}}
\put(12601,-4561){\makebox(0,0)[lb]{\smash{{\SetFigFont{5}{6.0}{\rmdefault}{\mddefault}{\updefault}{\color[rgb]{0,0,0}$C$}%
}}}}
\end{picture}%

%% file: Triangles.pstex_t
\begin{picture}(0,0)%
\includegraphics{Triangles.pstex}%
\end{picture}%
\setlength{\unitlength}{829sp}%
\begingroup\makeatletter\ifx\SetFigFont\undefined%
\gdef\SetFigFont#1#2#3#4#5{%
  \reset@font\fontsize{#1}{#2pt}%
  \fontfamily{#3}\fontseries{#4}\fontshape{#5}%
  \selectfont}%
\fi\endgroup%
\begin{picture}(26895,11766)(-5747,-10444)
\put(12331,-4471){\makebox(0,0)[lb]{\smash{{\SetFigFont{5}{6.0}{\rmdefault}{\mddefault}{\updefault}{\color[rgb]{0,0,0}$O_1$}%
}}}}
\put(1081,-4471){\makebox(0,0)[lb]{\smash{{\SetFigFont{5}{6.0}{\rmdefault}{\mddefault}{\updefault}{\color[rgb]{0,0,0}$O_3$}%
}}}}
\put(2476,-4471){\makebox(0,0)[lb]{\smash{{\SetFigFont{5}{6.0}{\rmdefault}{\mddefault}{\updefault}{\color[rgb]{0,0,0}$O_4$}%
}}}}
\put(-2069,-4471){\makebox(0,0)[lb]{\smash{{\SetFigFont{5}{6.0}{\rmdefault}{\mddefault}{\updefault}{\color[rgb]{0,0,0}$O_1$}%
}}}}
\put(14086,-4471){\makebox(0,0)[lb]{\smash{{\SetFigFont{5}{6.0}{\rmdefault}{\mddefault}{\updefault}{\color[rgb]{0,0,0}$O_2$}%
}}}}
\put(15481,-4471){\makebox(0,0)[lb]{\smash{{\SetFigFont{5}{6.0}{\rmdefault}{\mddefault}{\updefault}{\color[rgb]{0,0,0}$O_3$}%
}}}}
\put(16876,-4471){\makebox(0,0)[lb]{\smash{{\SetFigFont{5}{6.0}{\rmdefault}{\mddefault}{\updefault}{\color[rgb]{0,0,0}$O_4$}%
}}}}
\put(-314,-4471){\makebox(0,0)[lb]{\smash{{\SetFigFont{5}{6.0}{\rmdefault}{\mddefault}{\updefault}{\color[rgb]{0,0,0}$O_2$}%
}}}}
\end{picture}%

%% file: Fundamentals4.pstex_t
\begin{picture}(0,0)%
\includegraphics{Fundamentals4.pstex}%
\end{picture}%
\setlength{\unitlength}{1657sp}%
\begingroup\makeatletter\ifx\SetFigFont\undefined%
\gdef\SetFigFont#1#2#3#4#5{%
  \reset@font\fontsize{#1}{#2pt}%
  \fontfamily{#3}\fontseries{#4}\fontshape{#5}%
  \selectfont}%
\fi\endgroup%
\begin{picture}(12374,5916)(10160,-10669)
\put(12894,-7666){\makebox(0,0)[lb]{\smash{{\SetFigFont{5}{6.0}{\rmdefault}{\mddefault}{\updefault}{\color[rgb]{0,0,0}$O_2$}%
}}}}
\put(13591,-7666){\makebox(0,0)[lb]{\smash{{\SetFigFont{5}{6.0}{\rmdefault}{\mddefault}{\updefault}{\color[rgb]{0,0,0}$O_3$}%
}}}}
\put(14289,-7666){\makebox(0,0)[lb]{\smash{{\SetFigFont{5}{6.0}{\rmdefault}{\mddefault}{\updefault}{\color[rgb]{0,0,0}$O_4$}%
}}}}
\put(12016,-7666){\makebox(0,0)[lb]{\smash{{\SetFigFont{5}{6.0}{\rmdefault}{\mddefault}{\updefault}{\color[rgb]{0,0,0}$O_1$}%
}}}}
\put(20409,-7554){\makebox(0,0)[lb]{\smash{{\SetFigFont{5}{6.0}{\rmdefault}{\mddefault}{\updefault}{\color[rgb]{0,0,0}$O_4$}%
}}}}
\put(18339,-7554){\makebox(0,0)[lb]{\smash{{\SetFigFont{5}{6.0}{\rmdefault}{\mddefault}{\updefault}{\color[rgb]{0,0,0}$O_1$}%
}}}}
\put(19059,-7554){\makebox(0,0)[lb]{\smash{{\SetFigFont{5}{6.0}{\rmdefault}{\mddefault}{\updefault}{\color[rgb]{0,0,0}$O_2$}%
}}}}
\put(19734,-7554){\makebox(0,0)[lb]{\smash{{\SetFigFont{5}{6.0}{\rmdefault}{\mddefault}{\updefault}{\color[rgb]{0,0,0}$O_3$}%
}}}}
\put(21039,-7554){\makebox(0,0)[lb]{\smash{{\SetFigFont{5}{6.0}{\rmdefault}{\mddefault}{\updefault}{\color[rgb]{0,0,0}$O_3$}%
}}}}
\put(21714,-7554){\makebox(0,0)[lb]{\smash{{\SetFigFont{5}{6.0}{\rmdefault}{\mddefault}{\updefault}{\color[rgb]{0,0,0}$O_2$}%
}}}}
\put(17686,-7554){\makebox(0,0)[lb]{\smash{{\SetFigFont{5}{6.0}{\rmdefault}{\mddefault}{\updefault}{\color[rgb]{0,0,0}$O_2$}%
}}}}
\end{picture}%

%% file: MaslovComputation.pstex_t
\begin{picture}(0,0)%
\includegraphics{MaslovComputation.pstex}%
\end{picture}%
\setlength{\unitlength}{829sp}%
\begingroup\makeatletter\ifx\SetFigFont\undefined%
\gdef\SetFigFont#1#2#3#4#5{%
  \reset@font\fontsize{#1}{#2pt}%
  \fontfamily{#3}\fontseries{#4}\fontshape{#5}%
  \selectfont}%
\fi\endgroup%
\begin{picture}(8778,8166)(706,-7294)
\put(4321,-3166){\makebox(0,0)[lb]{\smash{{\SetFigFont{5}{6.0}{\rmdefault}{\mddefault}{\updefault}{\color[rgb]{0,0,0}$\beta$}%
}}}}
\put(6931,-4561){\makebox(0,0)[lb]{\smash{{\SetFigFont{5}{6.0}{\rmdefault}{\mddefault}{\updefault}{\color[rgb]{0,0,0}$\x$}%
}}}}
\put(7696,-3436){\makebox(0,0)[lb]{\smash{{\SetFigFont{5}{6.0}{\rmdefault}{\mddefault}{\updefault}{\color[rgb]{0,0,0}$Y-Z$}%
}}}}
\put(721,-556){\makebox(0,0)[lb]{\smash{{\SetFigFont{5}{6.0}{\rmdefault}{\mddefault}{\updefault}{\color[rgb]{0,0,0}$Y-Z$}%
}}}}
\put(2656,-1456){\makebox(0,0)[lb]{\smash{{\SetFigFont{5}{6.0}{\rmdefault}{\mddefault}{\updefault}{\color[rgb]{0,0,0}$\y$}%
}}}}
\put(6841,-736){\makebox(0,0)[lb]{\smash{{\SetFigFont{5}{6.0}{\rmdefault}{\mddefault}{\updefault}{\color[rgb]{0,0,0}$\z$}%
}}}}
\put(4501,-646){\makebox(0,0)[lb]{\smash{{\SetFigFont{5}{6.0}{\rmdefault}{\mddefault}{\updefault}{\color[rgb]{0,0,1}$\gamma$}%
}}}}
\put(6976,-6811){\makebox(0,0)[lb]{\smash{{\SetFigFont{5}{6.0}{\rmdefault}{\mddefault}{\updefault}{\color[rgb]{0,0,0}$X-Z$}%
}}}}
\put(4726,-2086){\makebox(0,0)[lb]{\smash{{\SetFigFont{5}{6.0}{\rmdefault}{\mddefault}{\updefault}{\color[rgb]{0,0,0}$X+Y$}%
}}}}
\put(6751,-2761){\makebox(0,0)[lb]{\smash{{\SetFigFont{5}{6.0}{\rmdefault}{\mddefault}{\updefault}{\color[rgb]{1,0,0}$\alpha$}%
}}}}
\put(3601,389){\makebox(0,0)[lb]{\smash{{\SetFigFont{5}{6.0}{\rmdefault}{\mddefault}{\updefault}{\color[rgb]{0,0,0}$X-Z$}%
}}}}
\end{picture}%

%% file: MaslovComputation5.pstex_t
\begin{picture}(0,0)%
\includegraphics{MaslovComputation5.pstex}%
\end{picture}%
\setlength{\unitlength}{1036sp}%
\begingroup\makeatletter\ifx\SetFigFont\undefined%
\gdef\SetFigFont#1#2#3#4#5{%
  \reset@font\fontsize{#1}{#2pt}%
  \fontfamily{#3}\fontseries{#4}\fontshape{#5}%
  \selectfont}%
\fi\endgroup%
\begin{picture}(24366,23016)(868,-22144)
\put(3601,389){\makebox(0,0)[lb]{\smash{{\SetFigFont{6}{7.2}{\rmdefault}{\mddefault}{\updefault}{\color[rgb]{0,0,0}$X-Z$}%
}}}}
\put(901,-511){\makebox(0,0)[lb]{\smash{{\SetFigFont{6}{7.2}{\rmdefault}{\mddefault}{\updefault}{\color[rgb]{0,0,0}$Y-Z$}%
}}}}
\put(10126,-9286){\makebox(0,0)[lb]{\smash{{\SetFigFont{6}{7.2}{\rmdefault}{\mddefault}{\updefault}{\color[rgb]{0,0,0}$X-Z$}%
}}}}
\put(9901,-736){\makebox(0,0)[lb]{\smash{{\SetFigFont{6}{7.2}{\rmdefault}{\mddefault}{\updefault}{\color[rgb]{0,0,0}$\z$}%
}}}}
\put(2251,-736){\makebox(0,0)[lb]{\smash{{\SetFigFont{6}{7.2}{\rmdefault}{\mddefault}{\updefault}{\color[rgb]{0,0,0}$\y$}%
}}}}
\put(6751,-3436){\makebox(0,0)[lb]{\smash{{\SetFigFont{6}{7.2}{\rmdefault}{\mddefault}{\updefault}{\color[rgb]{0,0,0}$X+Y$}%
}}}}
\put(4726,-736){\makebox(0,0)[lb]{\smash{{\SetFigFont{6}{7.2}{\rmdefault}{\mddefault}{\updefault}{\color[rgb]{0,0,1}$\gamma$}%
}}}}
\put(10126,-5236){\makebox(0,0)[lb]{\smash{{\SetFigFont{6}{7.2}{\rmdefault}{\mddefault}{\updefault}{\color[rgb]{1,0,0}$\alpha$}%
}}}}
\put(10576,-6136){\makebox(0,0)[lb]{\smash{{\SetFigFont{6}{7.2}{\rmdefault}{\mddefault}{\updefault}{\color[rgb]{0,0,0}$Y-Z$}%
}}}}
\put(7201,-6136){\makebox(0,0)[lb]{\smash{{\SetFigFont{6}{7.2}{\rmdefault}{\mddefault}{\updefault}{\color[rgb]{0,0,0}$\beta$}%
}}}}
\put(10126,-7936){\makebox(0,0)[lb]{\smash{{\SetFigFont{6}{7.2}{\rmdefault}{\mddefault}{\updefault}{\color[rgb]{0,0,0}$\x$}%
}}}}
\put(3601,-12211){\makebox(0,0)[lb]{\smash{{\SetFigFont{6}{7.2}{\rmdefault}{\mddefault}{\updefault}{\color[rgb]{0,0,0}$X-Z$}%
}}}}
\put(901,-13111){\makebox(0,0)[lb]{\smash{{\SetFigFont{6}{7.2}{\rmdefault}{\mddefault}{\updefault}{\color[rgb]{0,0,0}$Y-Z$}%
}}}}
\put(10126,-21886){\makebox(0,0)[lb]{\smash{{\SetFigFont{6}{7.2}{\rmdefault}{\mddefault}{\updefault}{\color[rgb]{0,0,0}$X-Z$}%
}}}}
\put(9901,-13336){\makebox(0,0)[lb]{\smash{{\SetFigFont{6}{7.2}{\rmdefault}{\mddefault}{\updefault}{\color[rgb]{0,0,0}$\z$}%
}}}}
\put(2251,-13336){\makebox(0,0)[lb]{\smash{{\SetFigFont{6}{7.2}{\rmdefault}{\mddefault}{\updefault}{\color[rgb]{0,0,0}$\y$}%
}}}}
\put(6751,-16036){\makebox(0,0)[lb]{\smash{{\SetFigFont{6}{7.2}{\rmdefault}{\mddefault}{\updefault}{\color[rgb]{0,0,0}$X+Y$}%
}}}}
\put(4726,-13336){\makebox(0,0)[lb]{\smash{{\SetFigFont{6}{7.2}{\rmdefault}{\mddefault}{\updefault}{\color[rgb]{0,0,1}$\gamma$}%
}}}}
\put(7201,-18736){\makebox(0,0)[lb]{\smash{{\SetFigFont{6}{7.2}{\rmdefault}{\mddefault}{\updefault}{\color[rgb]{0,0,0}$\beta$}%
}}}}
\put(10126,-20536){\makebox(0,0)[lb]{\smash{{\SetFigFont{6}{7.2}{\rmdefault}{\mddefault}{\updefault}{\color[rgb]{0,0,0}$\x$}%
}}}}
\put(17101,389){\makebox(0,0)[lb]{\smash{{\SetFigFont{6}{7.2}{\rmdefault}{\mddefault}{\updefault}{\color[rgb]{0,0,0}$X-Z$}%
}}}}
\put(14401,-511){\makebox(0,0)[lb]{\smash{{\SetFigFont{6}{7.2}{\rmdefault}{\mddefault}{\updefault}{\color[rgb]{0,0,0}$Y-Z$}%
}}}}
\put(23626,-9286){\makebox(0,0)[lb]{\smash{{\SetFigFont{6}{7.2}{\rmdefault}{\mddefault}{\updefault}{\color[rgb]{0,0,0}$X-Z$}%
}}}}
\put(23401,-736){\makebox(0,0)[lb]{\smash{{\SetFigFont{6}{7.2}{\rmdefault}{\mddefault}{\updefault}{\color[rgb]{0,0,0}$\z$}%
}}}}
\put(15751,-736){\makebox(0,0)[lb]{\smash{{\SetFigFont{6}{7.2}{\rmdefault}{\mddefault}{\updefault}{\color[rgb]{0,0,0}$\y$}%
}}}}
\put(20251,-3436){\makebox(0,0)[lb]{\smash{{\SetFigFont{6}{7.2}{\rmdefault}{\mddefault}{\updefault}{\color[rgb]{0,0,0}$X+Y$}%
}}}}
\put(18226,-736){\makebox(0,0)[lb]{\smash{{\SetFigFont{6}{7.2}{\rmdefault}{\mddefault}{\updefault}{\color[rgb]{0,0,1}$\gamma$}%
}}}}
\put(23626,-5236){\makebox(0,0)[lb]{\smash{{\SetFigFont{6}{7.2}{\rmdefault}{\mddefault}{\updefault}{\color[rgb]{1,0,0}$\alpha$}%
}}}}
\put(24076,-6136){\makebox(0,0)[lb]{\smash{{\SetFigFont{6}{7.2}{\rmdefault}{\mddefault}{\updefault}{\color[rgb]{0,0,0}$Y-Z$}%
}}}}
\put(20701,-6136){\makebox(0,0)[lb]{\smash{{\SetFigFont{6}{7.2}{\rmdefault}{\mddefault}{\updefault}{\color[rgb]{0,0,0}$\beta$}%
}}}}
\put(23626,-7936){\makebox(0,0)[lb]{\smash{{\SetFigFont{6}{7.2}{\rmdefault}{\mddefault}{\updefault}{\color[rgb]{0,0,0}$\x$}%
}}}}
\put(17101,-12211){\makebox(0,0)[lb]{\smash{{\SetFigFont{6}{7.2}{\rmdefault}{\mddefault}{\updefault}{\color[rgb]{0,0,0}$X-Z$}%
}}}}
\put(14401,-13111){\makebox(0,0)[lb]{\smash{{\SetFigFont{6}{7.2}{\rmdefault}{\mddefault}{\updefault}{\color[rgb]{0,0,0}$Y-Z$}%
}}}}
\put(23626,-21886){\makebox(0,0)[lb]{\smash{{\SetFigFont{6}{7.2}{\rmdefault}{\mddefault}{\updefault}{\color[rgb]{0,0,0}$X-Z$}%
}}}}
\put(23401,-13336){\makebox(0,0)[lb]{\smash{{\SetFigFont{6}{7.2}{\rmdefault}{\mddefault}{\updefault}{\color[rgb]{0,0,0}$\z$}%
}}}}
\put(15751,-13336){\makebox(0,0)[lb]{\smash{{\SetFigFont{6}{7.2}{\rmdefault}{\mddefault}{\updefault}{\color[rgb]{0,0,0}$\y$}%
}}}}
\put(18226,-13336){\makebox(0,0)[lb]{\smash{{\SetFigFont{6}{7.2}{\rmdefault}{\mddefault}{\updefault}{\color[rgb]{0,0,1}$\gamma$}%
}}}}
\put(23626,-17836){\makebox(0,0)[lb]{\smash{{\SetFigFont{6}{7.2}{\rmdefault}{\mddefault}{\updefault}{\color[rgb]{1,0,0}$\alpha$}%
}}}}
\put(24076,-18736){\makebox(0,0)[lb]{\smash{{\SetFigFont{6}{7.2}{\rmdefault}{\mddefault}{\updefault}{\color[rgb]{0,0,0}$Y-Z$}%
}}}}
\put(20701,-18736){\makebox(0,0)[lb]{\smash{{\SetFigFont{6}{7.2}{\rmdefault}{\mddefault}{\updefault}{\color[rgb]{0,0,0}$\beta$}%
}}}}
\put(23626,-20536){\makebox(0,0)[lb]{\smash{{\SetFigFont{6}{7.2}{\rmdefault}{\mddefault}{\updefault}{\color[rgb]{0,0,0}$\x$}%
}}}}
\put(10126,-17836){\makebox(0,0)[lb]{\smash{{\SetFigFont{6}{7.2}{\rmdefault}{\mddefault}{\updefault}{\color[rgb]{1,0,0}$\alpha$}%
}}}}
\put(10576,-18736){\makebox(0,0)[lb]{\smash{{\SetFigFont{6}{7.2}{\rmdefault}{\mddefault}{\updefault}{\color[rgb]{0,0,0}$Y-Z$}%
}}}}
\put(21601,-14461){\makebox(0,0)[lb]{\smash{{\SetFigFont{6}{7.2}{\rmdefault}{\mddefault}{\updefault}{\color[rgb]{0,0,0}$X+Y$}%
}}}}
\end{picture}%

%% file: ngons.tex
\section{Polygons with $n>3$ sides}


We now generalize the Heegaard diagram and the Heegaard triple as
follows.   Consider $\R^2$, with the $O_i$ marking the diagonal
line with half-integer coordinates as before. Choose $n$ sets
$\Las^1,\dots,\Las^n$ of lines, as follows. Choose an increasing set
of angles $\frac{\pi}{4}<\theta_1<\dots<\theta_n<\frac{5\pi}{4}$, and
let $\Las^j$ be parallel lines in $\R^n$ forming angle $\theta_j$ with
respect to the real axis $\R\times 0$, so that two consecutive lines 
are separated by some $O$ marking, and so that the intersections
between the various lines $\La^j_\ell$ within the set $\Las^j$ are in general position. Explicitly,
\[ \Las^j_\ell =
\{(\ell+\epsilon^j(\ell)+t\cos\theta_j,\ell+\epsilon^j(\ell)+t\sin\theta_j)\}_{t\in
  \R},\] where $\{\epsilon^j\colon \Z\to \R\}_{j=1}^n$ are
$G_m$-invariant functions with $\epsilon^j(\ell)<\OneHalf$, so that
$\epsilon^s(\ell)\neq \epsilon^t(\ell)$ for all $s\neq t$.  See
Figure~\ref{fig:HeegaardQuad} for a non-generic picture
(i.e. $\epsilon^j\equiv 0$ in the above formulas) in $\R^2$, when
$n=4$. Note that the picture is also rotated 90$\circ$: the line
through the $O$ markings should have slope $1$.

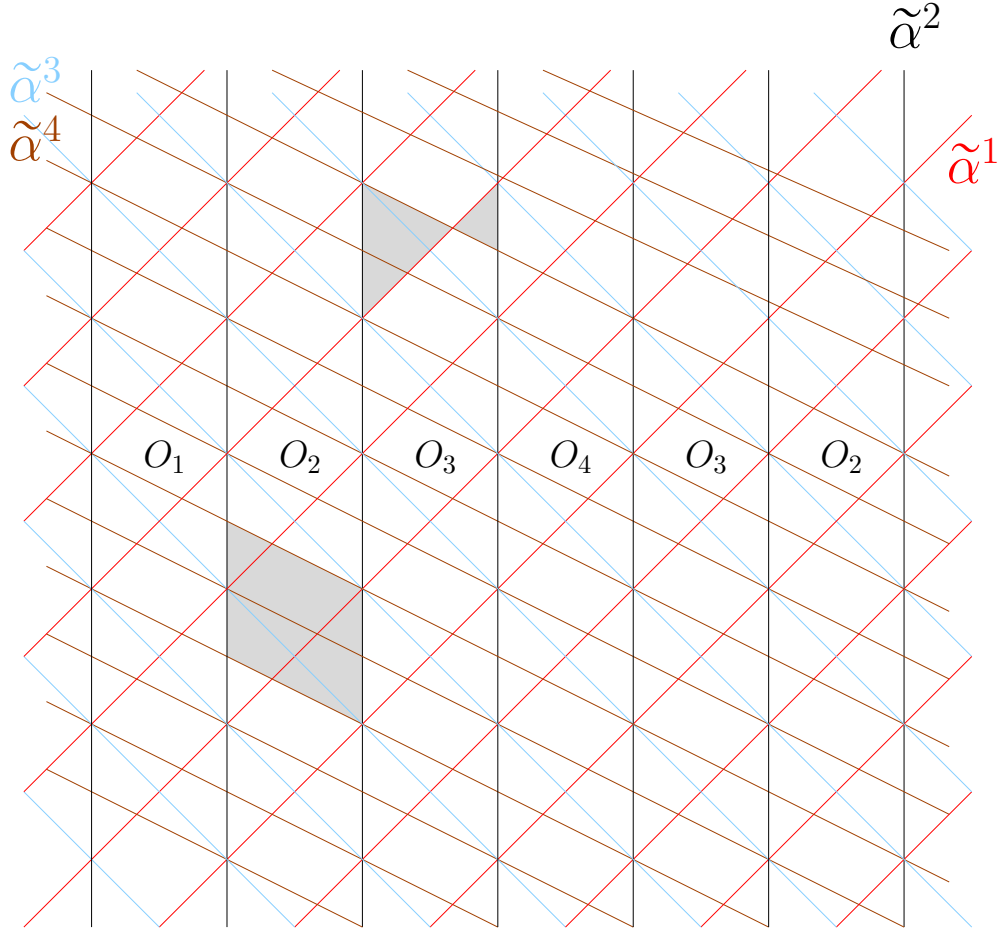
\begin{figure}[ht]
\input{HeegaardQuad5.pstex_t}
\caption{\label{fig:HeegaardQuad} {\bf{Heegaard quadruple.}}
This is the (unperturbed) multi-diagram with $n=4$ lifted to $\R^2$.
Two elements of ${\mathcal A}$ are shaded.}
\end{figure}

When $n=2$, a linear transformation carries this to the Heegaard
diagram from Section~\ref{sec:HeegPong}; and when $n=3$, a linear
transformation carries this to the Heegaard triple from
Section~\ref{sec:Triples}. As in Lemma~\ref{lem:WrapDiagram}, 
this is the Heegaard diagram for the $A_\infty$ actions on the wrapped Fukaya category.

Lemma~\ref{lem:EulerMeasure}, which was stated earlier for triangles,
actually holds for arbitrary $n$-gons:

\begin{lemma}
  \label{lem:EulerMeasureNGon}
  Let $\psi_n\in\pi_2(\x_1,\dots,\x_n)$ be a positive domain,
  then 
  \[ e(\psi_n)=\frac{k(n-2)}{4}+\frac{O_1(\psi_n)}{2}+\frac{O_m(\psi)}{2}.\]
\end{lemma}

\begin{proof}
  The argument used in the proof of Lemma~\ref{lem:EulerMeasure}
  when $n=3$ can be adapted to $n>2$, as follows.
  
  Let ${\mathcal A}$ denote the Abelian group of compactly supported
  $2$-chains which are required to be cornerless at all the intersections of
  $\La^j_\ell$ with $\L^{j'}_{\ell'}$, provided that $j,j'\neq 2$. 
  
  Fix Heegaard states $\x^{i,i+1}$ for $(\HD,\alphas^i,\alphas^{i+1})$
  and $\x^{n,1}$ for $(\HD,\alphas^{n},\alphas^{1})$.  Let ${\mathcal
    B}$ be the affine space for ${\mathcal A}$ with initial corners at
  the lifts of $\x^{3,4},\x^{4,5},\dots,\x^{n-1,n},\x^{n,1}$;
  i.e. there are no constraints placed on the intersections with
  $\alphas^2$. There are, once again, fundamental regions, which are
  rectangles (possibly with a self-intersection) formed now by two
  segments in $\Las^2$ and two other segments in $\Las^{j}$ and
  $\Las^{j'}$.  The regions are called {\em fundamental bigons} if the
  four segments are permuted by the ${\mathbb G}_m$ action, otherwise
  they are called {\em fundamental rectangles}.

  Like for the case of triangles, the fundamental rectangles and
  bigons generate ${\mathcal A}$; the function ${\widetilde
    e}=e-\frac{O_1}{2}-\frac{O_m}{2}$, which is defined on all
  (finite) $2$-chains, vanishes on ${\mathcal A}$.  Consider
  next ${\mathcal B}$, which is the set of (finite)
  $2$-chains with initial corners at
  $\x^{3,4},\x^{4,5},\dots,\x^{n-1,n}$, and terminal corner at
  $\x^{n,1}$.  We can find representatives of ${\mathcal B}/{\mathcal
    A}$, which are a union of $n-2$ triangles that miss $O_1$ and
  $O_m$. It is an easy computation to see that this representative has
  ${\widetilde e}=\frac{(n-2)}{k}$.

  The result is also true when $n=2$. We do not explicitly need it here,
  and we leave it to the reader to supply the details.
\end{proof}

Recall that a rigid holomorphic $n$-gon has $\Mas(\phi_n)=3-n$.

\begin{prop}
  \label{prop:NoNGons}
  There are no rigid, holomorphic Whitney $n$-gons with $n>3$;
  in particular, for any pseudo-holomorphic $n$-gon with $n>3$, 
  $\Mas(\psi_n)\geq 0$.
\end{prop}

\begin{proof}
  Combining Lemma~\ref{lem:EulerMeasureNGon} with
  Equation~\eqref{eq:IndexOfNgon}, we find that 
  \[ \Mas(\psi_n)=\frac{O_1(\psi_n)}{2}+
  \frac{O_m(\psi_n)}{2}+\#(\psi_n\cap \Delta).\] For
  pseudo-holomorphic $\psi_n$, all three terms on the right-hand-side
  are non-negative.
\end{proof}

\begin{proof}[of Theorem~\ref{thm:IdentifyPong}]
  The identification from Theorem~\ref{thm:IdentifyPong}
  consists of three statements: 
  \begin{itemize}
  \item The endomorphism algebra is isomorphic with the pong algebra,
    as a chain complex. This is Proposition~\ref{prop:IdentifyComplexes}.
  \item The composition law on the endomorphism algebra is identified 
    with the multiplication on the pong algebra.
    This is Proposition~\ref{prop:TrianglePong}.
  \item Like the pong algebra, the endomorphism algebra has vanishing
    $\mu_n$ with $n>3$. This is Proposition~\ref{prop:NoNGons}.
  \end{itemize}
\end{proof}

%% file: HeegaardQuad5.pstex_t
\begin{picture}(0,0)%
\includegraphics{HeegaardQuad5.pstex}%
\end{picture}%
\setlength{\unitlength}{2486sp}%
\begingroup\makeatletter\ifx\SetFigFont\undefined%
\gdef\SetFigFont#1#2#3#4#5{%
  \reset@font\fontsize{#1}{#2pt}%
  \fontfamily{#3}\fontseries{#4}\fontshape{#5}%
  \selectfont}%
\fi\endgroup%
\begin{picture}(9657,9186)(12856,-7723)
\put(14221,-3121){\makebox(0,0)[lb]{\smash{{\SetFigFont{14}{16.8}{\rmdefault}{\mddefault}{\updefault}{\color[rgb]{0,0,0}$O_1$}%
}}}}
\put(15571,-3121){\makebox(0,0)[lb]{\smash{{\SetFigFont{14}{16.8}{\rmdefault}{\mddefault}{\updefault}{\color[rgb]{0,0,0}$O_2$}%
}}}}
\put(16921,-3121){\makebox(0,0)[lb]{\smash{{\SetFigFont{14}{16.8}{\rmdefault}{\mddefault}{\updefault}{\color[rgb]{0,0,0}$O_3$}%
}}}}
\put(18271,-3121){\makebox(0,0)[lb]{\smash{{\SetFigFont{14}{16.8}{\rmdefault}{\mddefault}{\updefault}{\color[rgb]{0,0,0}$O_4$}%
}}}}
\put(19621,-3121){\makebox(0,0)[lb]{\smash{{\SetFigFont{14}{16.8}{\rmdefault}{\mddefault}{\updefault}{\color[rgb]{0,0,0}$O_3$}%
}}}}
\put(20971,-3121){\makebox(0,0)[lb]{\smash{{\SetFigFont{14}{16.8}{\rmdefault}{\mddefault}{\updefault}{\color[rgb]{0,0,0}$O_2$}%
}}}}
\put(21646,1064){\makebox(0,0)[lb]{\smash{{\SetFigFont{20}{24.0}{\rmdefault}{\mddefault}{\updefault}{\color[rgb]{0,0,0}$\Las^2$}%
}}}}
\put(12871,524){\makebox(0,0)[lb]{\smash{{\SetFigFont{20}{24.0}{\rmdefault}{\mddefault}{\updefault}{\color[rgb]{.53,.81,1}$\Las^3$}%
}}}}
\put(12871,-61){\makebox(0,0)[lb]{\smash{{\SetFigFont{20}{24.0}{\rmdefault}{\mddefault}{\updefault}{\color[rgb]{.63,.25,0}$\Las^4$}%
}}}}
\put(22231,-286){\makebox(0,0)[lb]{\smash{{\SetFigFont{20}{24.0}{\rmdefault}{\mddefault}{\updefault}{\color[rgb]{1,0,0}$\Las^1$}%
}}}}
\end{picture}%

%% file: further.tex
\section{Further}

There are several variants on the constructions considered here.  For
example, we could consider the cylinder $A$, equipped with $m$
vertical lines separated by $O$-markings. The endomorphism algebra, in
$\Sym^k(A)$ of the corresponding Lagrangians can be thought of as a
more symmetric, circular analogue of the pong algebra (i.e. one
without left and right walls), discovered by Manion and
Rouqier~\cite{ManionRouquier}.

In our notation, to each integer $m\geq 1$ and $1\leq k\leq m$, one
can consider a differential graded algebra $A(m,k)$ over
$\Field[v_1,\dots,v_m]$, which could be called the {\em asteroids
  algebra}, defined as follows.\footnote{The choice of terminology for
  our algebras can be taken as evidence for misspent youth.}  (This
algebra was fist considered in unpublished joint work of the first author with
Robert Lipshitz and Dylan Thurston, when consider knot Floer homology
for toroidal grid diagrams; compare~\cite{Gentle, PetkovaVertesi}.  It is also closely related to the
``differential graded nil Hecke algebras associated to the extended
affine symmetric groups'' of~\cite{ManionRouquier}.)  Consider the
circle $\R/m\Z$, equipped with with basepoints
$\{\OneHalf,\dots,\OneHalf+m-1\}$ corresponding to $v_1,\dots,v_m$. An
idempotent state corresponds to a $k$-element subset of
$\{1,\dots,m\}\subset \R/m\Z$; or, equivalently, a a $m\Z$-invariant
subset ${\widetilde S}\subset \Z$.  A ($m\Z$-lifted) partial
permutation, now, consists of such a subset ${\widetilde S}$, and a
map $f\colon {\widetilde S}\to \Z$ satisfying $f(x+m)=f(x)+m$. A {\em
  crossing} consists of a pair of integers $i<j$ so that
$f(i)>f(j)$. Once again, there is a Maslov grading that counts the
number of crossings. There is a multiplication map induced from
composition of $m\Z$-lifed partial permutations, which is set to $0$
if the Maslov grading of the product is smaller than the sum of the
Maslov gradings of the factors. Similarly, there is a differential
that resolves crossings, containing only those terms whose resolutions
have exactly one fewer crossing. Verifying that the result is a
differential graded algebra is straightforward, and slightly simpler
than the corresponding verification for the pong
algebra. (Compare~\cite[Section~4]{Pong}.)  Weights are at points in
$\frac{\OneHalf+\Z}{m\Z}$, with the weight of $a\in \OneHalf+\Z$ given
by
\[ \OneHalf (\#\{i\big| i<a<f(i)\} + \#\{i\big| f(i)<a<i\}).\]

For example, consider the $3\Z$-lifted partial permutation $f$ with
domain $\{1+3\Z,2+3\Z\}$, which is determined by 
\begin{equation}
  f(1)=6 \qquad{\text{and}}\qquad
  f(2)=1.
  \label{eq:Deff}
\end{equation}
 This has two crossings: the equivalence class of the pair of strands
 starting at $1$ and $2$; and the pair of strands starting at $-2$ and
 $3$. The weight vector is given by $(1/2,3/2,1)$. See
 Figure~\ref{fig:Asteroids} for a picture.

 Note also that when $m=4$ and $k=1$, this is the ``peculiar algebra''
 of Zibrowius~\cite{Zibrowius}.

\begin{figure}[ht]
\input{Asteroids.pstex_t}
\caption{\label{fig:Asteroids} {\bf{Asteroids diagram.}}
At the left, an asteroids diagram for $f$ from Equation~\eqref{eq:Deff}.
The other two pictures represent the terms in the differential of the first term,
both taken with multiplicity $v_2$.}
\end{figure}
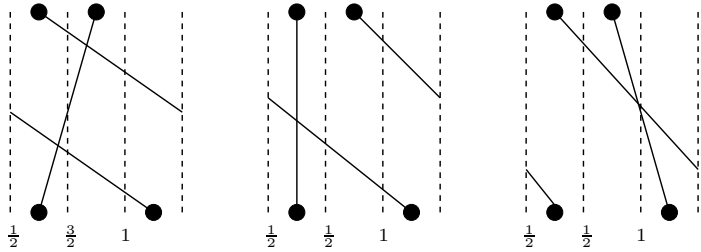

The Heegaard diagram for the wrapped Fukaya category (analogous
to the diagram from Section~\ref{sec:HeegPong})
is a quotient of $\R^2$ by a group of translations (rather
than the group of motions ${\mathbb G}_m$ considered above).
Proposition~\ref{prop:NoNGons} holds also in this case.

One can look at other configurations of Lagrangians in punctured
spheres.  For instance, for a linear chain of spheres, the
endomorphism algebra in the wrapped Fukaya category does have a higher
multiplication (and indeed it can be computed, for example, with the
methods of~\cite{Pong}). 

Finally, building on~\cite{Pong}, the pong algebra can be thought of
as governing the bordered invariants for certain types of Heegaard
diagrams associated to tangles; see~\cite{NextPong}; compare
also~\cite{InvPair, HolKnot}. In light of this, the present work can
be thought of as analogous to Auroux's interpretation of bordered
Floer homology~\cite{Auroux}.

%% file: Asteroids.pstex_t
\begin{picture}(0,0)%
\includegraphics{Asteroids.pstex}%
\end{picture}%
\setlength{\unitlength}{1184sp}%
\begingroup\makeatletter\ifx\SetFigFont\undefined%
\gdef\SetFigFont#1#2#3#4#5{%
  \reset@font\fontsize{#1}{#2pt}%
  \fontfamily{#3}\fontseries{#4}\fontshape{#5}%
  \selectfont}%
\fi\endgroup%
\begin{picture}(14598,5102)(436,-5888)
\put(451,-5761){\makebox(0,0)[lb]{\smash{{\SetFigFont{7}{8.4}{\rmdefault}{\mddefault}{\updefault}{\color[rgb]{0,0,0}$\OneHalf$}%
}}}}
\put(1651,-5761){\makebox(0,0)[lb]{\smash{{\SetFigFont{7}{8.4}{\rmdefault}{\mddefault}{\updefault}{\color[rgb]{0,0,0}$\frac{3}{2}$}%
}}}}
\put(2851,-5761){\makebox(0,0)[lb]{\smash{{\SetFigFont{7}{8.4}{\rmdefault}{\mddefault}{\updefault}{\color[rgb]{0,0,0}$1$}%
}}}}
\put(11251,-5761){\makebox(0,0)[lb]{\smash{{\SetFigFont{7}{8.4}{\rmdefault}{\mddefault}{\updefault}{\color[rgb]{0,0,0}$\OneHalf$}%
}}}}
\put(12451,-5761){\makebox(0,0)[lb]{\smash{{\SetFigFont{7}{8.4}{\rmdefault}{\mddefault}{\updefault}{\color[rgb]{0,0,0}$\frac{1}{2}$}%
}}}}
\put(8251,-5761){\makebox(0,0)[lb]{\smash{{\SetFigFont{7}{8.4}{\rmdefault}{\mddefault}{\updefault}{\color[rgb]{0,0,0}$1$}%
}}}}
\put(5851,-5761){\makebox(0,0)[lb]{\smash{{\SetFigFont{7}{8.4}{\rmdefault}{\mddefault}{\updefault}{\color[rgb]{0,0,0}$\OneHalf$}%
}}}}
\put(7051,-5761){\makebox(0,0)[lb]{\smash{{\SetFigFont{7}{8.4}{\rmdefault}{\mddefault}{\updefault}{\color[rgb]{0,0,0}$\frac{1}{2}$}%
}}}}
\put(13651,-5761){\makebox(0,0)[lb]{\smash{{\SetFigFont{7}{8.4}{\rmdefault}{\mddefault}{\updefault}{\color[rgb]{0,0,0}$1$}%
}}}}
\end{picture}%